\documentclass{amsart}
\usepackage[utf8]{inputenc}

\title[Gelfand-Kirillov dimension of representations of $\GL_n$]{Gelfand-Kirillov dimension of Representations of $\GL_n$ over a non-archimedean local field}
\author{Kenta Suzuki}
\address{
  Department of Mathematics,
  Massachusetts Institute of Technology,
  Cambridge,
  MA 02142-4307 USA
}
\email{kjsuzuki@mit.edu}

\date{August 2022}

\usepackage{amsmath, amssymb, amsthm,multirow,mathtools,verbatim,tikz,tikz-cd,placeins,graphicx,realboxes,pdflscape,afterpage,capt-of,color,colonequals,stmaryrd,mathrsfs,hyperref}
\usepackage{appendix}
\usepackage[shortlabels]{enumitem}

\numberwithin{equation}{section}

\newtheorem{thm}{Theorem}[section]

\newtheorem{lemma}[thm]{Lemma}
\newtheorem{cor}[thm]{Corollary}
\newtheorem{prop}[thm]{Proposition}

\theoremstyle{definition}
\newtheorem*{rmk}{Remark}

\newtheorem*{defn}{Definition}
\newtheorem{example}[thm]{Example}

\newcommand{\C}{{\mathbb C}}
\newcommand{\Z}{{\mathbb Z}}
\newcommand{\Q}{{\mathbb Q}}
\newcommand{\cO}{{\mathcal O}}
\newcommand{\sO}{{\mathscr O}}

\newcommand{\cN}{{\mathcal N}}
\newcommand{\bP}{{\mathbb P}}
\newcommand{\bK}{{\mathbf K}}

\newcommand{\fA}{{\mathfrak A}}
\newcommand{\one}{{\mathbf 1}}
\newcommand{\p}{{\mathfrak p}}
\newcommand{\fg}{{\mathfrak g}}

\newcommand{\SL}{{\mathrm{SL}}}
\newcommand{\GL}{{\mathrm{GL}}}
\newcommand{\PGL}{{\mathrm{PGL}}}
\DeclareMathOperator{\Gal}{Gal}

\DeclareMathOperator{\Ind}{Ind}
\DeclareMathOperator{\cInd}{c-Ind}
\DeclareMathOperator{\nInd}{n-Ind}
\DeclareMathOperator{\supp}{supp}

\begin{document}
\maketitle
\begin{abstract}
    We calculate the asymptotic behavior of the dimension of the fixed vectors of $\pi$ with respect to compact open subgroups $1+ M_n(\p^N)\subset\GL_n(F)$ for $\pi$ an admissible representation of $\GL_n(F)$, and $F$ a nonarchimedean local field. Such dimensions can be calculated by germs of the character of $\pi$. We also make some observations on how those dimensions behave under instances of Langlands functoriality, such as the Jacquet-Langlands correspondence and cyclic base change, where relations between characters are known.
\end{abstract}

\tableofcontents

\section{Introduction}\label{section-introduction}

Smooth representations of $\GL_n(F)$, for a nonarchimedean local field $F$ such as $\Q_p$ or $\mathbb F_p((t))$, is of great interest in number theory, for instance due to the local Langlands correspondence. However, most such representations are infinite-dimensional, so finding a finite measure of their size is important. One such measure is the Gelfand-Kirillov dimension:

\begin{defn}
Let $\pi$ be a (complex) admissible representation of $\GL_n(F)$. The \emph{Gelfand-Kirillov dimension} (henceforth, GK-dimension) of $\pi$ is a real number $r\ge0$, denoted $\dim_{GK}(\pi)$, such that
\[
\dim(\pi^{K_N})\sim q^{rN},
\]
where $K_N\colonequals 1+M_n(\p_F^N)$ is a compact open subgroup of $\GL_n(F)$ and $q$ is the order of the residue field $\sO_F/\p_F$. Here, for functions $f(N)$ and $g(N)$, denote $f(N)\sim g(N)$ if there exists real numbers $0<a<b$ such that for sufficiently large integers $N>0$, the inequality $af(N)<g(N)<bf(N)$ is satisfied.
\end{defn}

In fact, Harish-Chandra and Howe's local character formula says this function $\dim(\pi^{K_N})$ is a polynomial in $q^{N-1}$ for $N$ large (see Corollary~\ref{fixed-poly}), which we call the \emph{growth polynomial} of $\pi$. Now, $\dim_{GK}(\pi)$ is exactly the degree of the growth polynomial (and, in particular, is an integer). A more ambitious goal would be to calculate the entire growth polynomial. The coefficients of the growth polynomial contain important information about the representation, as they behave well under instances of Langlands functoriality (see Section~\ref{section-functoriality}.)

Irreducible representations of $\GL_n(F)$ were classified by Bernstein and Zelevinsky \cite{BZ2} in terms of multisegments---combinatorial data---of supercuspidal representations (those corresponding to irreducible representations under the Langlands correspondence). Roughly, arbitrary representations are obtained by inducing from parabolic subgroups, and supercuspidal representations are those not obtained from such a procedure.

For supercuspidal representations $\rho$ of $\GL_n(F)$, it is a result of Howe \cite{howe} that $\dim_{GK}(\rho)=\frac{n(n-1)}2$, and in fact, the leading term of the growth polynomial is known (see Proposition~\ref{cusp-asymp}):
\[\dim(\rho^{K_N})=(1+o(1))[n!]_qq^{\frac{n(n-1)}2(N-1)},\]
where $[n!]_q\colonequals [1]_q\cdots[n]_q=\frac{q-1}{q-1}\cdots\frac{q^n-1}{q-1}$ is the $q$-factorial. This was extended to generic representations by Rodier \cite{rodier2}.

We further extend these results, and calculate the leading term of the growth polynomial for \emph{arbitrary} admissible representations of $\GL_n(F)$, based on the combinatorial datum of the Bernstein-Zelevinsky classification \cite{BZ2}:
\begin{thm}\label{intro-thm}
Let $\pi$ be an arbitrary admissible irreducible representation of $\GL_n(F)$. Then, there exists a multisegment $a=\{\Delta_1,\dots,\Delta_m\}$ such that $\rho=\langle a\rangle$, where $\Delta_i=[\rho_i,\nu^{r_i-1}\rho_i]$ for supercuspidal representations $\rho_i$ of $\GL_{n_i}(F)$. Then,
\[
\dim(\pi^{K_N})=(1+o(1))\frac{[n!]_q}{[r_1!]_{q^{n_1}}\cdots[r_m!]_{q^{n_m}}}(q^{N-1})^{\frac12(n^2-n_1r_1^2-\cdots-n_mr_m^2)},
\]
so in particular,
\[
\dim_{GK}(\pi)=\frac12(n^2-n_1r_1^2-\cdots-n_mr_m^2).
\]
\end{thm}

In fact, we are able to express the growth polynomial of arbitrary admissible representations of $\GL_n(F)$ \emph{exactly} in terms of growth polynomials for supercuspidal representations. As a consequence, we deduce that $\dim(\pi^{K_N})$ is an \emph{integer-coefficient} polynomial in $q^{N-1}$ (see Corollary~\ref{int-coeff}), generalizing a result of Howe \cite{howe}.

The difficulty is that although every irreducible representation of $G=\GL_n(F)$ appears as a subquotient of some induced representation $\nInd_P^G(\rho)$ with $\rho$ a supercuspidal representation of the parabolic subgroup $P\subset G$, such representations may have \emph{multiple} irreducible subquotients, and we can a priori only detect the maximal GK-dimension (see Lemma~\ref{exact-seq}). We identify exactly the irreducible subquotient of $\nInd_P^G(\rho)$ with maximal GK-dimension (see Lemma~\ref{main-lemma}) and for each irreducible representation $\pi$ find a pair $(P,\rho)$ such that $\pi$ has maximal GK-dimension among the irreducible subquotients of $\nInd_P^G(\rho)$. Here, $\rho$ may no longer be supercuspidal, but are products of representation of the form $\langle\Delta\rangle$ for some segment $\Delta$, whose growth polynomial is calculated in Subsection~\ref{subsection-delta}.


The paper is structured as follows: In section~\ref{section-terms} we gather the notation used throughout the paper. In section~\ref{section-arbitrary} we prove the main result, the calculation of the leading term of the growth polynomial for arbitrary irreducible representations of $\GL_n(F)$. In section~\ref{section-more-examples} we calculate some examples, corresponding to those multisegments with particularly simple combinatorial structure. In section~\ref{section-local-char} we look at Howe's local character formula, and reproduce a proof that indeed $\dim(\pi^{K_N})$ is a polynomial in $q^{N-1}$. Finally, in section~\ref{section-functoriality} we explore how growth polynomials behave under some instances of Langlands functoriality, such as Jacquet-Langlands, base change, and automorphic induction.

We also briefly consider the growth polynomial for representations of $\SL_n(F)$ in appendix~\ref{section-sl-n}, which we deduce from the theory for $\GL_n(F)$. In appendix~\ref{appendix-cusp} we provide a more direct proof of $\dim_{GK}(\rho)=\frac{n(n-1)}2$ for $\rho$ supercuspidal, based on Bushnell and Kutzko's classification \cite{Bush-Kutz}. Finally, in appendix~\ref{appendix-level-0} we provide explicit calculations for growth polynomials of level zero representations of $\GL_2$ and $\GL_3$, based on the techniques developed in appendix~\ref{appendix-cusp}.


\section{Terminology and Conventions}\label{section-terms}

Throughout the paper, $F$ is a nonarchimedean local field of residual characteristic $p$. That is, $F$ is a finite extension of $\Q_p$ or $\mathbb F_p((T))$. Let $\sO_F$ denote the ring of integers of $F$, and let $\p_F\subset\sO_F$ denote the unique maximal ideal. Moreover, let $q_F$ denote the size of the residue field $k_F\colonequals\sO_F/\p_F$. Subscripts are dropped whenever $F$ is clear from context. 

Note that $F$ is assumed to have characteristic zero in Sections~\ref{section-local-char}~and~\ref{section-functoriality}.

Fix a nontrivial character $\psi\colon F\to\C^\times$ of level zero, i.e., $\psi$ is trivial on $\p_F$ but not on $\cO_F$.

By a \emph{representation} of $\GL_n(F)$, we always mean a complex, admissible, smooth representations. That is, we will only consider representations $(\pi,V)$ such that $\pi\colon\GL_n(F)\to\GL_\C(V)$ is continuous with respect to the discrete topology on $\GL_\C(V)$, and such that for any compact open subgroup $K\subset\GL_n(F)$ the space of fixed vectors
\[
V^K\colonequals \{v\in V:\pi(k)v=v\text{ for any }k\in K\}
\]
is finite-dimensional. The notation $\pi$ is used for arbitrary admissible smooth representations of $\GL_n(F)$, while $\rho$ is reserved for supercuspidal representations. We will denote induction by $\Ind$, compact induction by $\cInd$, and normalized induction by $\nInd$ (i.e., $\nInd_H^G(\pi)\colonequals \Ind_H^G(\delta_G^{1/2}\delta_H^{-1/2}\otimes\pi)$).

Let $\mathcal R$ be the Grothendieck ring of representations of $\GL$, i.e., the free abelian group generated by representations of $\GL_n$, modulo $\phi-\pi-\sigma$ for any exact sequence
\[
0\to\pi\to\phi\to\sigma\to0
\]
of representations of $\GL_n$. It becomes a graded commutative ring by the product defined in Section~\ref{section-arbitrary} (see also \cite[Thm~1.9]{BZ2}).

We define a series of compact open subgroups of $\GL_n(F)$, by $K_0\colonequals \GL_n(\sO_F)$, the maximal compact subgroup, and for $N>0$, by
\[
K_N\colonequals 1+M_n(\p_F^N)=\ker(K_0\to\GL_n(\sO_F/\p_F^N)).
\]
The subgroups $\{K_N\}_{N\ge0}$ form a fundamental system of neighborhoods of $1\in\GL_n(F)$.

\begin{defn}
Let $\pi$ be a representation of $\GL_n(F)$. Then, the \emph{growth polynomial} of $\pi$ is the polynomial $G_\pi(X)\in\C[X]$ such that for $N\gg0$,
\[
\dim(\pi^{K_N})=G_\pi(q^{N-1}).
\]
Such a polynomial exists by Corollary~\ref{fixed-poly} and Lemma~\ref{exact-seq}. Denote by $\dim_{GK}(\pi)$ the degree of $G_\pi$.
\end{defn}

Finally, let $\mathrm{rec}_F\colon \mathbf{Irr}(\GL_n(F))\to \mathcal G_n(F)$ be the local Langlands correspondence, where $\mathbf{Irr}(\GL_n(F))$ is the set of irreducible representations of $\GL_n(F)$ and $\mathcal G_n(F)$ is the set of semisimple Deligne representations of the Weil group $\mathcal W_F$ (a close relative to the absolute Galois group $\Gal(\overline F/F)$ of $F$).


\section{Reduction to supercuspidal representations}\label{section-arbitrary}

Notation and terminology in this section mostly follows \cite{BZ2}, so $\nu\colon F^\times\to\C^\times$ is the character $\nu(g)\colonequals|\det(g)|$, and a \emph{segment} is a set of supercuspidal representations of the form $[\rho,\nu^{r-1}\rho]\colonequals\{\rho,\nu\rho,\dots,\nu^{r-1}\rho\}$ for some supercuspidal representation $\rho$ and integer $r>0$. By \cite[Thm~6.1]{BZ2} irreducible representations of $\GL_n$ are classified by \emph{multisegments}---multisets of segments.

A key construction in Bernstein and Zelevinsky's classification is inducing from parabolic subgroups $P\subset G$. Conjugacy classes of $P$ correspond to partitions of $n$. Indeed, a partition $n=n_1+\cdots+n_r$, denoted $\lambda$, corresponds to
\[
P_\lambda\colonequals \begin{pmatrix}
\GL_{n_1}&*&&*\\
&\GL_{n_2}&*\\
&&\ddots&*\\
0&&&\GL_{n_r}
\end{pmatrix}\subset G.
\]

Let $\pi_1,\dots,\pi_r$ be admissible representations of $\GL_{n_1},\dots,\GL_{n_r}$, respectively. Following the notation in \cite{BZ2}, we denote
\[
\pi_1\times\cdots\times\pi_r\colonequals \nInd_{P_\lambda}^G(\pi_1\boxtimes\cdots\boxtimes\pi_r),
\]
where $\pi_1\boxtimes\cdots\boxtimes\pi_r$ is inflated via the surjection $P_\lambda\to\GL_{n_1}\times\cdots\times\GL_{n_r}$, which quotients out the unipotent radical $N_\lambda\subset P_\lambda$.

The following lemma on the size of the double cosets $P_\lambda\backslash G/K_N$ is useful:

\begin{lemma}\label{parabolic-size}
Let $\lambda$ be the partition $n=n_1+\cdots+n_r$. Then,
\[
|P_\lambda\backslash G/K_N|=\begin{bmatrix}n\\ n_1,\cdots,n_r\end{bmatrix}_q(q^{N-1})^{(n^2-n_1^2-\cdots-n_r^2)/2}
\]
for $N>0$, where $\begin{bmatrix}n\\n_1,\dots,n_r\end{bmatrix}_q\colonequals \frac{[n!]_q}{[n_1!]_q\cdots[n_r!]_q}$, and $[n!]_q$ is the $q$-analogue of $n!$, defined in Section~\ref{section-introduction}.
\end{lemma}
\begin{proof}
Since $G=P_\lambda K_0$, we have
\[
P_\lambda\backslash G/K_N\cong P_\lambda(\sO)\backslash K_0/K_N\cong P_\lambda(\sO/\p^N)\backslash\GL_n(\sO/\p^N),
\]
the partial flag variety over $\sO/\p^N$ corresponding to the partition $\lambda$. Now, the subgroup
\[
\begin{pmatrix}
I_{n_1}&0&\cdots&0\\
\p/\p^N&I_{n_2}&&\vdots\\
\vdots&&\ddots&0\\
\p/\p^N&\cdots&\p/\p^N&I_{n_r}
\end{pmatrix}\subset \GL_n(\sO/\p^N)
\]
acts simply on $P_\lambda(\sO/\p^N)\backslash\GL_n(\sO/\p^N)$ from the right, with quotient $P_\lambda(k)\backslash\GL_n(k)$, so
\begin{align*}
|P_\lambda\backslash G/K_N|&=|P_\lambda(k)\backslash\GL_n(k)|\cdot (q^{N-1})^{(n^2-n_1^2-\cdots-n_r^2)/2}\\
&=\begin{bmatrix}n\\ n_1,\cdots,n_r\end{bmatrix}_q(q^{N-1})^{(n^2-n_1^2-\cdots-n_r^2)/2}.
\end{align*}
Here, the last equality is because the quotient $P_\lambda(k)\backslash\GL_n(k)$ is the partial flag variety over $k$, and the projective space $\bP_k^{n-1}$ has size $[n]_q=\frac{q^n-1}{q-1}$.
\end{proof}
\begin{rmk}
Here, $\dim(N_\lambda)=\frac12(n^2-n_1^2-\cdots-n_r^2)$, where $N_\lambda\subset P_\lambda$ is the nilpotent radical.
\end{rmk}

The following Mackey-type formula is also useful, and will be used repeatedly throughout the paper:

\begin{lemma}\label{fixed}
Let $G$ be a locally profinite group, $H\subset G$ a closed subgroup, and $K\subset G$ a compact open subgroup. Let $S\subset G$ be a set of representatives for $H\backslash G/K$. Then, for a smooth representation $\lambda$ of $H$,
\begin{align*}
    (\cInd_H^G\lambda)^K&\xrightarrow{\sim}\bigoplus_{g\in S}\lambda^{H\cap gKg^{-1}}\\
    f&\mapsto (f(g))_{g\in S}.
\end{align*}
\end{lemma}
\begin{proof}
We will first show well-definedness. Let $f\in(\cInd_H^G\lambda)^K$. For each $g\in S$ and $h\in H\cap gKg^{-1}$, we have:
\[
f(g)=f(h\cdot g\cdot g^{-1}h^{-1}g)=\lambda(h)f(g),
\]
so $f(g)\in\lambda^{H\cap gKg^{-1}}$. Moreover, there is a subset $T\subset S$ such that $\supp f=\sqcup_{g\in T}HgK$, where each $HgK$ is open (it is a union of $K$-cosets). Thus, \[H\backslash\supp f=\bigsqcup_{g\in T}H\backslash HgK,\] so since $H\backslash\supp f$ is compact, $T\subset S$ must be finite. That is, the image lies in the direct sum.

We will construct an inverse. Let $\mathbf a=(a_g)_{g\in S}\in\bigoplus_{g\in S}\lambda^{H\cap gKg^{-1}}$ be arbitrary. For any $u\in G=HSK$, there is a $h\in H$, $g\in S$, and $k\in K$ such that $u=hgk$. Let $f_{\mathbf a}(u)\colonequals \lambda(h)a_g$. We claim $f_{\mathbf a}\colon G\to\lambda$ is well-defined. Indeed, if $u=hgk=h'gk'$, then
\[
\lambda(h)a_g=\lambda(h')\lambda((h')^{-1}h)a_g=\lambda(h')a_g,
\]
since $(h')^{-1}h=gk'k^{-1}g^{-1}\in H\cap gKg^{-1}$. Now, $f_{\mathbf a}\in(\Ind_H^G\lambda)^K$ is clear. Finally, $f_{\mathbf a}$ is compactly supported modulo $H$, since
\[
H\backslash \supp f=\bigcup_{g\in T}H\backslash HgK
\]
where $T=\{g\in S:a_g\ne 0\}$ is finite, and each $H\backslash HgK$ is the continuous image of the compact set $K$.
\end{proof}

Now, the growth polynomials (defined in Section~\ref{section-terms}) of induced representations of the form $\pi_1\times\cdots\times\pi_r$ are readily calculated:

\begin{lemma}\label{induced-dim}
Let $n=n_1+\cdots+n_r$ be a partition and let $\pi_1,\dots,\pi_r$ be admissible representations of $\GL_{n_1},\dots,\GL_{n_r}$, respectively. Then, 
\[
G_{\pi_1\times\cdots\times\pi_r}(X)=\begin{bmatrix}n\\ n_1,\cdots,n_r\end{bmatrix}_qX^{(n^2-n_1^2-\cdots-n_r^2)/2}\prod_{i=1}^rG_{\pi_i}(X),
\]
so in particular,
\[
\dim_{GK}(\pi_1\times\cdots\times\pi_r)=\frac12(n^2-n_1^2-\cdots-n_r^2)+\sum_{i=1}^r\dim_{GK}(\pi_i).
\]
\end{lemma}
\begin{proof}
Let $\pi=\pi_1\times\cdots\times\pi_r$. The normalization may be ignored since the character $\delta_G^{1/2}\delta_P^{-1/2}$ is trivial on a sufficiently small neighborhood of $P$. By Lemma~\ref{fixed} we have
\[
\rho^{K_N}=\bigoplus_{g\in P\backslash G/K_N}(\rho_1\boxtimes\cdots\boxtimes\rho_r)^{P(F)\cap gK_Ng^{-1}}.
\]
Now, by Iwasawa decomposition $G=BK_0=PK_0$, so the representatives $g\in P\backslash G/K_N$ can be chosen so that $g\in K_0$. Then, $gK_Ng^{-1}=K_N$, so
\begin{align*}
\pi^{K_N}&=\big((\pi_1\boxtimes\cdots\boxtimes\pi_r)^{P(F)\cap K_N}\big)^{\oplus|P\backslash G/K_N|}\\
&=\big((\pi_1\boxtimes\cdots\boxtimes\pi_r)^{K_N\times\cdots\times K_N}\big)^{\oplus|P\backslash G/K_N|}\\
&=\big(\pi_1^{K_N}\otimes\cdots\otimes\pi_r^{K_N}\big)^{\oplus|P\backslash G/K_N|}.
\end{align*}
Now the result follows from Lemma~\ref{parabolic-size}.
\end{proof}

Moreover, growth polynomials behave well with short exact sequences:

\begin{lemma}\label{exact-seq}
Let
\[
0\to\pi\to\phi\to\sigma\to0
\]
be an exact sequence of admissible representations of $G$. Then,
\[
G_\phi(X)=G_\pi(X)+G_\sigma(X),
\]
so in particular, $\dim_{GK}(\phi)=\max\{\dim_{GK}(\pi),\dim_{GK}(\sigma)\}$.
\end{lemma}
\begin{proof}
Follows immediately from \cite[Section~2.3]{Bush-Henn}~Corollary~1, that $(-)^{K_N}$ is an exact functor.
\end{proof}

\begin{rmk}
A nice formalization of Lemma~\ref{induced-dim} and Lemma~\ref{exact-seq}, together with Corollary~\ref{fixed-poly} is the following: the map from the Grothendieck ring of representations of $\GL$
\begin{align*}
    \mathscr I\colon\mathcal R&\to\C[X,X^{-1}]\\
    \pi\in\mathrm{Rep}(\GL_n)&\mapsto [n!]_q^{-1}X^{-\frac{n^2-n}2}G_\pi(X)
\end{align*}
is a well-defined ring homomorphism. That is, $\mathscr I$ normalizes $\dim(\pi^{K_N})$ by the size of the full flag variety over $\sO/\p^N$.
\end{rmk}



\subsection{The growth polynomial of $\langle\Delta\rangle$}\label{subsection-delta} We calculate the growth polynomial for representations of the form $\langle\Delta\rangle$ for some segment $\Delta$, which is of particular importance since they generate the Grothendieck ring $\mathcal R$ \cite[Cor~7.5]{BZ2}. Thus, growth polynomial for arbitrary representations can be deduced from that of $\langle\Delta\rangle$, via Lemma~\ref{induced-dim} and Lemma~\ref{exact-seq}, as carried out in the next subsection.

We first review the theory of Bernstein-Zelevinsky derivatives, as defined in \cite[Section~3.5]{BZ1}:

\begin{defn}
Let $M_n\subset\GL_n$ be the \emph{mirabolic subgroup}, the group of matrices whose last row is $(0,\cdots,0,1)$. For an admissible representation $(\tau,U)$ of $M_n$ and $0<i\le n$, define the \emph{$i$-th derivative} $\tau^{(i)}$ of $\tau$, which is a representation of $\GL_{n-i}$, as
\[
\tau^{(k)}\colonequals U/(\theta_i(h)v-\pi(h)v:v\in U,h\in N_i),
\]
where $N_i$ consists of elements of the form
\[
h=\left(\begin{array}{@{}c|c@{}}
\begin{matrix}\begin{matrix}
1 \\
&\ddots &\\
&& 1
\end{matrix}\\ \hline
\begin{matrix}
&\\
&\\
&
\end{matrix}\end{matrix} &
\begin{matrix}
* & * &*\\
\vdots & \vdots&\vdots\\
a_{n-i}&*&*\\
1 & a_{n-i+1}& *\\
&\ddots&a_{n-1}\\
 && 1
\end{matrix}
\end{array}\right)
\]
and $\theta_i(h)\colonequals\psi(a_{n-i+1}+\cdots+a_{n-1})$, where $\psi$ is the fixed character of $F$ as defined in Section~\ref{section-terms}. Here, the top-left block has size $(n-i)\times(n-i)$, and the bottom-right block has size $i\times i$. Moreover, for representations of $\GL_n$ define the $i$-th derivative as the $i$-th derivative of the restriction to $M_n\subset\GL_n$.
\end{defn}
\begin{defn}
Let $\pi$ be a representation of $\GL_n(F)$. Then, let the \emph{total derivative} be the formal sum in the Grothendieck ring
\[
\mathscr{D}(\pi)\colonequals\pi+\pi^{(1)}+\cdots+\pi^{(n)}\in\mathcal R.
\]
Extending $\mathscr D$ to $\mathcal R$ by linearity defines a ring homomorphism $\mathcal R\to\mathcal R$, by \cite[Lem~4.5]{BZ1}.
\end{defn}
Now, the following calculations of total derivatives are of particular interest to us:
\begin{lemma}[Bernstein-Zelevinsky~{\cite[Cor~4.14~(a)]{BZ1}}]\label{der-calc1}
Let $\pi$ be a representation of $\GL_m$, and let $\varphi$ be a representation of the parabolic subgroup $P_{(n-1)+1}\subset\GL_n$. Then, the derivatives of
\[
\sigma\colonequals\nInd_{P_{m+(n-1)+1}}^{P_{(m+n-1)+1}}(\pi\boxtimes\varphi)
\]
are given by, for $i>0$,
\[
\sigma^{(i)}=\pi\times\varphi^{(i)}.
\]
\end{lemma}

\begin{lemma}[Zelevinsky~{\cite[Thm~3.5]{BZ2}}]\label{der-calc2}
Let $\Delta=[\rho,\nu^{r-1}\rho]$ be a segment, where $\rho$ is a supercuspidal representation. Then,
\[
\mathscr{D}(\langle\Delta\rangle)=\langle\Delta\rangle+\langle\Delta^-\rangle,
\]
where $\Delta^-\colonequals[\rho,\nu^{r-2}\rho]$.
\end{lemma}

Now, we calculate the growth polynomial for representations of the form $\langle\Delta\rangle$ for some segment $\Delta$, heavily utilizing the fact that they remain irreducible when restricted to $M_n$ \cite[Rmk~3.6]{BZ2}:

\begin{prop}\label{delta-dim}
Let $\rho$ be a supercuspidal representation of $\GL_{n_1}(F)$, and let $\Delta=[\rho,\nu^{r-1}\rho]$, so $\langle\Delta\rangle$ is a representation of $\GL_n(F)$, with $n=n_1r$. Then,
\begin{align*}
G_{\langle\Delta\rangle}(X)&=[r!]_{q^{n_1}}^{-1}\begin{bmatrix}n\\n_1,\dots,n_1\end{bmatrix}_qX^{\frac{n_1(n_1-1)r(r-1)}2}G_\rho(X)^r\\
&=\frac{[n!]_q}{[r!]_{q^{n_1}}}X^{\frac{n(n-r)}2}+(\text{lower order terms}),
\end{align*}
so in particular,
\[
\dim_{GK}\langle\Delta\rangle=\frac{n(n-r)}2.
\]
\end{prop}
\begin{proof}
Throughout the proof, we abbreviate the parabolic subgroup $P_\lambda$ as the partition $\lambda$.

There is a $M_n$-homomorphism
\begin{align*}
i\colon \langle\Delta\rangle&\hookrightarrow\langle\Delta^-\rangle\times\nu^{r-1}\rho=\nInd_{(n-n_1)+n_1}^{n}(\langle\Delta^-\rangle\boxtimes\nu^{r-1}\rho)\\&\xrightarrow{res}\sigma\colonequals\nInd_{(n-n_1)+(n_1-1)+1}^{(n-1)+1}(\langle\Delta^-\rangle\boxtimes\nu^{r-1}\rho|_{(n_1-1)+1}),
\end{align*}
where the inclusion $\langle\Delta\rangle\hookrightarrow\langle\Delta^-\rangle\times\nu^{r-1}\rho$ is the $G$-homomorphism defined by the characterization of $\langle\Delta\rangle$ as the unique irreducible submodule of $\rho\times\cdots\times\nu^{r-1}\rho$ \cite[Section~3.1]{BZ2}, and $res$ is the surjection given by restricting functions $f\colon G\to\langle\Delta^-\rangle\boxtimes\nu^{r-1}\rho$ to $P_{(n-1)+1}$. We calculate their derivatives using Lemma~\ref{der-calc1} and Lemma~\ref{der-calc2}:
\begin{align*}
\mathscr{D}(\langle\Delta\rangle)&=\langle\Delta^-\rangle+\langle\Delta\rangle\\
\mathscr{D}(\langle\Delta^-\rangle\times\nu^{r-1}\rho)&=\langle\Delta^{--}\rangle+\langle\Delta^-\rangle+\langle\Delta^{--}\rangle\times\nu^{r-1}\rho+\langle\Delta^-\rangle\times\nu^{r-1}\rho\\
\mathscr{D}(\sigma)&=\langle\Delta^-\rangle+\sigma.
\end{align*}
Here $i$ is nonzero, since both $\langle\Delta\rangle$ and $\sigma$ have $n_1$-th derivative $\langle\Delta^-\rangle$, but the $n_1$-th derivative $\langle\Delta^-\rangle+\langle\Delta^{--}\rangle\times\nu^{r-1}\rho$ of $\langle\Delta^-\rangle\times\nu^{r-1}\rho$ only contains the irreducible representation $\langle\Delta^-\rangle$ with multiplicity one. Now, since $\langle\Delta\rangle$ is an irreducible $M_n$-representation by \cite[Cor~7.9]{BZ2}, the map $i$ must be injective.

Moreover, if $\tau$ is the cokernel of $i$, then its total derivative is
\begin{align*}
\mathscr D(\tau)&=\mathscr{D}(\sigma)-\mathscr{D}(\langle\Delta\rangle)\\
&=\sigma-\langle\Delta\rangle\\
&=\tau.
\end{align*}
That is, $\tau$ has no derivatives of degree $>0$, so by \cite[Section~3.5]{BZ2} we have $\tau=0$. Thus, $i$ is an isomorphism of $M_n$-representations.



The overall strategy is as follows: Use the isomorphism $i$ to calculate the dimension of $\langle\Delta\rangle^{K_N\cap M_n}\supseteq \langle\Delta\rangle^{K_N}$. In fact, we can show \[\langle\Delta\rangle^{K_N\cap M_n}=\langle\Delta\rangle^{K_N\cap\GL_{n-1}}=\langle\Delta\rangle^{K_N\cap {^t}M_n},\]
so
\[
\langle\Delta\rangle^{K_N}=\langle\Delta\rangle^{K_N\cap M_n}\cap \langle\Delta\rangle^{K_N\cap {^t}M_n}=\langle\Delta\rangle^{K_N\cap M_n},
\]
which allows us to calculate $\dim(\langle\Delta\rangle^{K_N})$.


First, for $N_1\ge N$ let
\[
J_{N_1}\colonequals\left(\begin{array}{@{}c|c@{}}
\begin{matrix}
\\
K_N\cap\GL_{n-1}\\
{}
\end{matrix} &
\begin{matrix}
\p^{N_1}\\
\vdots\\
\p^{N_1}
\end{matrix}\\\hline
0\quad \cdots\quad 0&1
\end{array}\right)\subset M_n
\]
be compact open subgroups which approximate the subgroup $K_N\cap\GL_{n-1}$ of $M_n$, with $J_N=K_N\cap M_n$. Here,
\begin{equation}\label{symmetric-fixed}
\langle\Delta\rangle^{K_N\cap\GL_{n-1}}=\bigcup_{N_1\ge N}\langle\Delta\rangle^{J_{N_1}},
\end{equation}
where $\supseteq$ is obvious and $\subseteq$ is from smoothness of the representation $\langle\Delta\rangle|_{M_n}$. This allows us to use Lemma~\ref{fixed} on the groups $J_{N_1}$.

For each $N_1\ge N$ we have
\begin{equation}\label{double-coset-calc}
P_{(n-n_1)+(n_1-1)+1}\backslash P_{(n-1)+1}/J_{N_1}\cong P_{(n-n_1)+(n_1-1)}\backslash \GL_{n-1}/K_N\cap\GL_{n-1}
\end{equation}
has size $\begin{bmatrix}n-1\\n_1-1\end{bmatrix}_qq^{n(n_1-1)(N-1)}$ by Lemma~\ref{parabolic-size}, which notably does not depend on the choice of $N_1$. Since $i$ is a $M_n$-isomorphism,
\begin{align}\label{integer-coeff-eq}\dim(\langle\Delta\rangle^{J_{N_1}})&=\dim(\sigma^{J_{N_1}})\nonumber\\
&=|P_{(n-n_1)+(n_1-1)+1}\backslash P_{(n-1)+1}/J_{N_1}|\dim\big(\langle\Delta^-\rangle^{K_N}\otimes(\nu^{r-1}\rho)^{K_N}\big)\\
&=\begin{bmatrix}n-1\\n_1-1\end{bmatrix}_qq^{n(n_1-1)(N-1)}\dim(\langle\Delta^-\rangle^{K_N})\dim(\rho^{K_N}),\nonumber\end{align}
where the second equality is by a similar argument as in Lemma~\ref{induced-dim}. By the independence on the choice of $N_1$, the chain of inclusions
\[
\langle\Delta\rangle^{J_N}\subseteq \langle\Delta\rangle^{J_{N+1}}\subseteq\langle\Delta\rangle^{J_{N+2}}\subseteq\cdots
\]
must all be equalities, so in particular, \eqref{symmetric-fixed} tells us that $\langle\Delta\rangle^{K_N\cap M_N}=\langle\Delta\rangle^{K_N\cap \GL_{n-1}}$.

On the other hand, by \cite[Thm~7.3]{BZ} we have $\langle\Delta\rangle\circ\iota\cong\langle\Delta\rangle^\vee\cong\langle[\nu^{1-r}\rho^\vee,\rho^\vee]\rangle$, where $\iota(g)\colonequals {^t}g^{-1}$ is the involution on $\GL_n(F)$ corresponding to the nontrivial automorphism of the Dynkin diagram $A_{n-1}$. Thus, for any subgroup $H\subset M_n$ we have
\[
\langle\Delta\rangle^{{^t}H}\cong(\langle\Delta\rangle^\vee)^H.
\]
Symmetrically, all the same arguments, with $\rho$ replaced by $\nu^{1-r}\rho^\vee$, tells us that $\langle\Delta\rangle^{K_N\cap {^t}M_N}=\langle\Delta\rangle^{K_N\cap \GL_{n-1}}$.

Since the groups $K_N\cap M_n$ and $K_N\cap {^t}M_n$ generate $K_N$,
\[
\langle\Delta\rangle^{K_N}=\langle\Delta\rangle^{K_N\cap M_n}\cap\langle\Delta\rangle^{K_N\cap {^t}M_n}=\langle\Delta\rangle^{K_N\cap M_n}.
\]
But this dimension was calculated in equation~\ref{integer-coeff-eq}, thus giving us a recursive formula for $\langle\Delta\rangle^{K_N}$.

By downward induction on $r$, and observing that
\[
\begin{bmatrix}n-1\\n_1-1\end{bmatrix}_q=\frac{[n_1]_q}{[n]_q}\cdot\begin{bmatrix}n\\n_1\end{bmatrix}_q=[r]_{q^{n_1}}^{-1}\begin{bmatrix}n_1r\\n_1\end{bmatrix}_q,
\]
we obtain
\[
\dim(\langle\Delta\rangle^{K_N})=[r!]_{q^{n_1}}^{-1}\begin{bmatrix}n\\n_1,\dots,n_1\end{bmatrix}_q(q^{N-1})^{\frac{n_1(n_1-1)r(r-1)}{2}}\dim(\rho^{K_N})^r.
\]
The final asymptotic is by Proposition~\ref{cusp-asymp}.
\end{proof}

\begin{rmk}
In the notation of the remark after Lemma~\ref{exact-seq}, Proposition~\ref{delta-dim} becomes:
\[
\mathscr I_{\langle\Delta\rangle}(X)=[r!]_{q^{n_1}}^{-1}X^{-\frac{r-1}2n}\mathscr I_\rho(X)^r.
\]
\end{rmk}

As a corollary, we generalize a result of Howe \cite{howe} that $G_\rho(X)\in\Z[X]$, for $\rho$ supercuspidal:
\begin{cor}\label{int-coeff}
Let $\pi$ be an arbitrary admissible representation of $\GL_n(F)$. Then, the growth polynomial $G_\pi(X)$ has integer coefficients.
\end{cor}
\begin{proof}
The statement holds for $\pi$ supercuspidal, by Corollary~\ref{fixed-poly} and \cite[Thm~3]{howe}, which states $c_\cO(\pi)\in\Z$ for any nilpotent orbit $\cO$.

The general case follows from Proposition~\ref{delta-dim}, Lemma~\ref{induced-dim}, and the fact that the Grothendieck ring $\mathcal R$ of representations of $\GL$ is generated as a ring by representations of the form $\langle\Delta\rangle$ (see \cite[Cor~7.5]{BZ2}). That $[r!]_{q^{n_1}}^{-1}\begin{bmatrix}n\\n_1,\dots,n_1\end{bmatrix}_q$ in Proposition~\ref{delta-dim} is an integer is more readily seen from the recursion~\eqref{integer-coeff-eq}.
\end{proof}

\begin{rmk}
The weaker statement, that $G_\pi(X)\in\Q[X]$, is immediate since $\dim(\pi^{K_N})$ is always an integer.
\end{rmk}

\subsection{Growth polynomial of arbitrary admissible representations} Although Proposition~\ref{delta-dim} allows for the calculation of growth polynomials for arbitrary representations, explicitly expressing representations in terms of $\langle\Delta\rangle$ is difficult in general. Nonetheless, Theorem~\ref{main-thm} calculates the leading term in terms of the Bernstein-Zelevinsky classification.

The following lemma allows us to relate the leading term of the growth polynomial of the representation $\langle a\rangle$ to that of the easier-understood representation $\pi(a)$.

\begin{lemma}\label{main-lemma}
For any multiset $a$ of segments,
\[
\dim(\langle a\rangle^{K_N})=(1+o(1))\dim(\pi(a)^{K_N})
\]
so in particular, $\dim_{GK}(\pi(a))=\dim_{GK}(\langle a\rangle)$.
\end{lemma}
\begin{proof}
First of all, we claim that if $a<b$ (i.e., $a$ can be obtained from $b$ via a chain of elementary operations) then $\dim_{GK}\pi(a)<\dim_{GK}\pi(b)$. Indeed, by Lemma~\ref{induced-dim} it suffices to check that for linked segments $\Delta$ and $\Delta'$,
\[
\dim(\langle\Delta\cap\Delta'\rangle\times\langle\Delta\cup\Delta'\rangle)<\dim(\langle\Delta\rangle\times\langle\Delta'\rangle).
\]
Let $\Delta=[\rho,\nu^{r-1}\rho]$ and $\Delta'=[\nu^s\rho,\nu^{s+t-1}\rho]$ for some $0<s\le r<s+t$ and $\rho$ a supercuspidal representation of $\GL_{n_1}(F)$. We hope to show
\[
\dim_{GK}(\langle\Delta\cap\Delta'\rangle\times\langle\Delta\cup\Delta'\rangle)=(t+s)(r-s)n_1^2+((t+s)^2+(r-s)^2)\frac{n_1(n_1-1)}2.
\]
is less than
\[
\dim_{GK}(\langle\Delta\rangle\times\langle\Delta'\rangle)=rtn_1^2+(r^2+t^2)\frac{n_1(n_1-1)}2.
\]
This reduces to the fact that $(t+s)^2+(t-s)^2>r^2+t^2$, i.e., the convexity of $x^2$.

Now by \cite[Thm~7.1]{BZ2}, in the Grothendieck ring $\mathcal R$ of representations of $\GL$ we have
\[
\pi(a)=\langle a\rangle+\sum_{b<a}m(b;a)\langle b\rangle
\]
for some nonnegative integers $m(b;a)$. For $b<a$ we have $\dim_{GK}\langle b\rangle\le \dim_{GK}\pi(b)<\dim_{GK}\pi(a)$, so
\[
\dim(\pi(a)^{K_N})=\dim(\langle a\rangle^{K_N})+\sum_{b<a}m(b;a)\dim(\langle b\rangle^{K_N})=(1+o(1))\dim(\langle a\rangle^{K_N}).\qedhere
\]\end{proof}
We finally arrive at our main theorem:
\begin{thm}\label{main-thm}
Let $\pi$ be an arbitrary admissible irreducible representation of $\GL_n(F)$. Then, there exists a multisegment $a=\{\Delta_1,\dots,\Delta_m\}$ such that $\rho=\langle a\rangle$, where $\Delta_i=[\rho_i,\nu^{r_i-1}\rho_i]$ for supercuspidal representations $\rho_i$ of $\GL_{n_i}(F)$. Then,
\[
G_\pi(X)=\frac{[n!]_q}{[r_1!]_{q^{n_1}}\cdots[r_m!]_{q^{n_m}}}X^{\frac12(n^2-n_1r_1^2-\cdots-n_mr_m^2)}+(\text{lower order terms}),
\]
so in particular, $\dim_{GK}(\pi)=\frac12(n^2-n_1r_1^2-\cdots-n_mr_m^2)$.
\end{thm}
\begin{proof}
Combining Lemma~\ref{induced-dim}, Lemma~\ref{main-lemma}, and Proposition~\ref{delta-dim}, we obtain:
\begin{align*}
    \dim(\pi^{K_N})&=(1+o(1))\dim((\langle\Delta_1\rangle\times\cdots\times\langle\Delta_m\rangle)^{K_N})\\
    &=(1+o(1))\begin{bmatrix}n\\n_1r_1,\dots,n_mr_m\end{bmatrix}_q(q^{N-1})^{\frac{n^2-n_1^2r_1^2-\cdots-n_m^2r_m^2}2}\prod_{i=1}^m\dim(\langle\Delta_i\rangle^{K_N})\\
    &=(1+o(1))\frac{[n!]_q}{[r_1!]_{q^{n_1}}\cdots[r_m!]_{q^{n_m}}}(q^{N-1})^{\frac12(n^2-n_1r_1^2-\cdots-n_mr_m^2)}.\qedhere
\end{align*}
\end{proof}

\begin{example}\label{bz-example}
Consider representations with support $\{\rho,\nu\rho,\nu\rho,\nu^2\rho\}$, as in \cite[Ex~11.4]{BZ2}, with $\rho$ a supercuspidal representation of $\GL_{n_1}(F)$. Equivalently, they are irreducible subquotients of $\rho\times\nu\rho\times\nu\rho\times\nu^2\rho$ \cite[Thm~7.1]{BZ2}. Then, the Hasse diagram for multisegments supported on it is (where the partial order is induced from elementary operations):
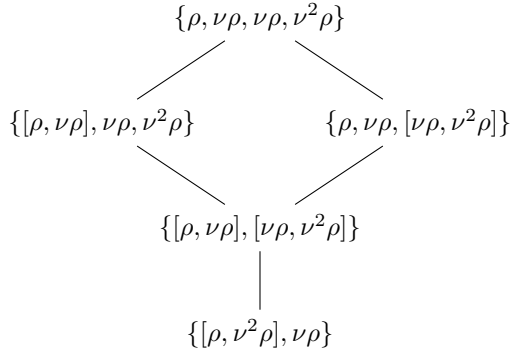
\begin{center}\begin{tikzpicture}[scale=.7]
  \node (max) at (0,2) {$\{\rho,\nu\rho,\nu\rho,\nu^2\rho\}$};
  \node (l) at (-3,0) {$\{[\rho,\nu\rho],\nu\rho,\nu^2\rho\}$};
  \node (r) at (3,0) {$\{\rho,\nu\rho,[\nu\rho,\nu^2\rho]\}$};
  \node (zero) at (0,-2) {$\{[\rho,\nu\rho],[\nu\rho,\nu^2\rho]\}$};
  \node (min) at (0,-4) {$\{[\rho,\nu^2\rho],\nu\rho\}$};
  \draw (l)--(max)--(r);
  \draw (l)--(zero)--(r);
  \draw (zero)--(min);
\end{tikzpicture}\end{center}
and the corresponding GK-dimensions of $\langle a\rangle$ (which, due to Lemma~\ref{main-lemma} and its proof respects the Hasse diagram) are:
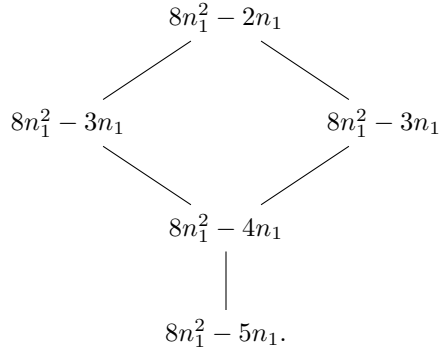
\begin{center}\begin{tikzpicture}[scale=.7]
  \node (max) at (0,2) {$8n_1^2-2n_1$};
  \node (l) at (-3,0) {$8n_1^2-3n_1$};
  \node (r) at (3,0) {$8n_1^2-3n_1$};
  \node (zero) at (0,-2) {$8n_1^2-4n_1$};
  \node (min) at (0,-4) {$8n_1^2-5n_1$.};
  \draw (l)--(max)--(r);
  \draw (l)--(zero)--(r);
  \draw (zero)--(min);
\end{tikzpicture}\end{center}

\end{example}

\begin{example}[Generalized Steinberg representations]
For a supercuspidal representation $\rho$ of $\GL_{n_1}$ let $\Delta=[\rho,\nu^{r-1}\rho]$, and consider the generalized Steinberg representation $\langle\Delta\rangle^t=\langle\rho,\dots,\nu^{r-1}\rho\rangle$ of $\GL_n$. Then, by Theorem~\ref{main-thm},
\[
G_{\langle\Delta\rangle^t}(X)=[n!]_qX^{\frac12(n-1)n}+(\text{lower order terms}).
\]
Proposition~\ref{example-prop} gives a more explicit formula.
\end{example}

\begin{example}[Generic representations]
More generally, if $\pi$ is a generic representation of $\GL_n(F)$, then
\[
G_\pi(X)=[n!]_qX^{\frac12(n-1)n}+(\text{lower order terms}).
\]
so the asymptotic formula is independent of the supercuspidal support. In fact, generic representations are characterized by being maximal (with respect to the order induced by elementary operations) in the poset of representations with a fixed supercuspidal support \cite[Thm~9.7]{BZ2}. Thus, $\dim_{GK}(\pi)=\frac{n^2-n}2$ if and only if $\pi$ is generic, recovering Rodier's result \cite{rodier2}.
\end{example}

\section{More examples}\label{section-more-examples}
In this section, we will only consider $\rho$-rigid representations, for some fixed supercuspidal representation $\rho$:
\begin{defn}
An irreducible representation $\pi$ of $\GL_n(F)$ is \emph{$\rho$-rigid} for some supercuspidal representation $\rho$ if the supercuspidal support of $\pi$ only contains representations of the form $\nu^i\rho$ for some $i\in\Z$.
\end{defn}
Only considering such representations suffice, by \cite[Prop~8.6]{BZ2} and Lemma~\ref{induced-dim}.

An explicit formula for the growth polynomial is available in some cases, when the structure of the lattice of multisegments is simple:
\begin{prop}\label{example-prop}
Let a $\rho$-rigid multiset of segments $a$ with support of size $s$ be such that any two of its segments have an empty intersection, where $\rho$ is a supercuspidal representation of $\GL_{n_1}(F)$ (so $n=n_1s$). Then,
\[
G_{\langle a\rangle}(X)=\begin{bmatrix}n\\n_1,\dots,n_1\end{bmatrix}_qX^{\frac{s(s-1)}2n_1^2}G_\rho(X)^{s}\sum_{b\le a}(-1)^{s-|b|}\prod_{i=1}^m[r_i!]_{q^{n_1}}^{-1}X^{-\frac{r_i(r_i-1)}2n_1},
\]
where $b=\{\Delta_1,\dots,\Delta_m\}\le a$, with $\Delta_i$ a segment of size $r_i$.
\end{prop}
\begin{proof}
By \cite[Prop~9.13]{BZ2}, in the Grothendieck ring we have the identity
\[
\langle a\rangle=\sum_{b\le a}(-1)^{|a|-|b|}\pi(b),
\]
so by Lemma~\ref{induced-dim} and Lemma~\ref{exact-seq} we have
\begin{align*}
G_{\langle a\rangle}(X)&=\sum_{b\le a}(-1)^{|a|-|b|}G_{\pi(b)}(X)\\
&=\sum_{b\le a}(-1)^{|a|-|b|}\begin{bmatrix}n\\r_1n_1,\dots,r_mn_m\end{bmatrix}_qX^{\frac{s^2-r_1^2-\dots-r_m^2}{2}n_1^2}\prod_{i=1}^mG_{\langle\Delta_i\rangle}(X).
\end{align*}
Now, applying Proposition~\ref{delta-dim} and simplifying gives the desired result.
\end{proof}
\begin{example}
Some example calculations are:
\begin{align*}
    G_{\langle[\rho,\nu\rho]\rangle^t}(X)&=\begin{bmatrix}2n_1\\n_1\end{bmatrix}_qX^{n_1^2}G_\rho(X)^2\{1-[2!]_{q^{n_1}}^{-1}X^{-n_1}\},\\
    G_{\langle[\rho,\nu^2\rho]\rangle^t}(X)&=\begin{bmatrix}3n_1\\n_1,n_1,n_1\end{bmatrix}_qX^{3n_1^2}G_\rho(X)^3\{1-2[2!]_{q^{n_1}}^{-1}X^{-n_1}+[3!]_{q^{n_1}}^{-1}X^{-3n_1}\}.
\end{align*}
\end{example}

\begin{example}
In the notation of Proposition~\ref{example-prop}, when $n_1=1$ and $\rho=\chi$ is a character,
\[
G_{\langle a\rangle}(X)=\sum_{b\le a}(-1)^{s-|b|}\begin{bmatrix}s\\r_1,\dots,r_m\end{bmatrix}_qX^{\frac12(s^2-\sum_{i=1}^mr_i^2)}.
\]
\end{example}

\begin{rmk}
More generally, for a $\rho$-rigid multiset of segments $a$ with support of size $s$,
\[
G_{\langle a\rangle}(X)=Q(X)\cdot G_\rho(X)^s,
\]
for some polynomial $Q(X)$ determined by the structure of the poset of multisegments on the support of $a$ (as in Example~\ref{bz-example}). In Proposition~\ref{example-prop}, said poset is a Boolean lattice.
\end{rmk}



\section{General theory of local character expansions}\label{section-local-char}
Section~\ref{section-arbitrary} reduces the calculation of growth polynomials of representations to that of supercuspidal representations. To calculate the growth polynomials of supercuspidal representations, we use the \emph{local character expansion}, which tells us the germ of a character.

Throughout this section, \emph{assume $F$ has characteristic zero}. Let $G$ be the group of $F$-rational points of a connected reductive algebraic $F$-group (e.g., $\GL_n(F)$ or $\GL_m(D)$, for a division algebra $D$ over $F$.) Also let $\fg$ be its Lie algebra, so there is an exponential map
\[
\exp_G\colon\fg_0\to G_1,
\]
for some neighborhood $\fg_0$ of $0\in\fg$ and neighborhood $G_1$ of $1\in G$. In fact, $\exp_G$ is bijective for sufficiently small neighborhoods $\fg_0\ni0$ and $G_1\ni1$.

\begin{example}
For $G=\GL_n(F)$ (resp., $\GL_m(D)$), we have $\fg=M_n(F)$ (resp., $M_m(D)$), and for $\fg_0\colonequals M_n(p\sO_F)$ (resp., $M_m(p\sO_D)$) and $G_1\colonequals 1+\fg_0$, there is an isomorphism of groups
\begin{align*}
\fg_0&\to G_1\\
X&\mapsto 1+X+\frac{X^2}2+\cdots.
\end{align*}
Unlike in the archimedean case, $\sum_{n\ge0}\frac{X^n}{n!}$ does not converge on the entire Lie algebra.
\end{example}

Let $Ch_\pi$ denote the character associated to an admissible representation $(\pi,V)$ of $G$. That is, for $f\in C_c^\infty(G)$ a compactly supported function on $G$,
\[
\mathrm{tr}\,\pi(f)=\int_GCh_\pi(g)f(g)dg,
\]
where $\pi(f)=\int_G f(g)\pi(g)dg\colon V\to V$ has finite rank so the trace is well-defined (restrict to a finite-dimensional subspace $\pi(f)(V)\subset W$ such that $\pi(f)(W)\subset W$). Note that $\pi(g)\colon V\to V$ for $g\in G$ need not have finite rank so its ``literal" trace is ill-defined.

For a compactly supported function $f\in C_c^\infty(\fg)$, define the \emph{Fourier transform} as
\[
\widehat f(X)\colonequals \int_\fg f(Y)\cdot\psi(B(X,Y))dY,
\]
where $B(-,-)$ is the Killing form on $\fg$ and $\psi$ is the fixed character of $F$ as defined in Section~\ref{section-terms}.

\begin{defn}
Let $\cN_G$ be the set of nilpotent $\mathrm{Ad}\ G$-orbits in $\fg$, called the nilpotent orbits of $G$. It is partially ordered with respect to the order $\cO_1\le\cO_2$ when $\overline{\cO}_1\subseteq \overline{\cO}_2$, where $\overline{\cO}_1\subset G$ is the Zariski closure of $\overline{\cO}_1$.
\end{defn}

\begin{example}\label{gln-nilp}
For $G=\GL_n(F)$, there is a 1-1 correspondence between the partially ordered sets:
\begin{enumerate}
\item\label{item1} $\cN_G$, the set of nilpotent orbits of $G$;
    \item\label{item2} equivalence classes of parabolic subgroups $P\subset G$; and
    \item\label{item3} partitions of $n$,
\end{enumerate}
where the correspondence between \eqref{item1} and \eqref{item2} reverses order.
Indeed, for $\cO\in\cN_G$ let $x\in\cO$ and consider the parabolic subgroup fixing the flag defined by $x$. Conversely, for any parabolic $P\subset G$, let $N\subset P$ be the unipotent radical and let $\mathfrak n\subset\fg$ be its Lie algebra, so $N=I_n+\mathfrak n$. Then, consider the unique nilpotent orbit $\cO$ such that $\cO\cap\mathfrak n\subset\mathfrak n$ is open. For each partition $\lambda$ of $n$, denote by $\cO_\lambda\in\cN_G$ the corresponding nilpotent orbit. Finally, the correspondence between \eqref{item2} and \eqref{item3} is demonstrated in the beginning of Section~\ref{section-arbitrary}.There is a similar correspondence for $G'=\GL_m(D)$.

Note that when $\lambda$ is the partition $n=n_1+\cdots+n_r$, then for $N_\lambda\subset P_\lambda$ the nilpotent radical,
\[
\dim(\cO_\lambda)=2\dim(N_\lambda)=\sum_{i\ne j}n_in_j=n^2-n_1^2-\cdots-n_r^2.
\]
\end{example}

\begin{lemma}
Any nilpotent orbit $\cO$ in $G$ has a $G$-invariant Radon measure $\mu_\cO$. Moreover, letting ${\widehat\mu}_\cO(f)\colonequals \mu_\cO(\widehat f)$ for $f\in C_c^\infty(\fg)$ (where $\widehat f$ denotes the Fourier transform with respect to the self-dual measure), the distribution $\widehat\mu_\cO$ is represented by a locally integrable function, again denoted $\widehat\mu_\cO$.
\end{lemma}

\begin{example}
For $f\in C_c^\infty(\fg)$ we have $\mu_{\{0\}}(f)=f(0)$, so $\widehat\mu_{\{0\}}=1$.
\end{example}

Now, recall Howe and Harish-Chandra's local character expansion \cite{harish-chandra}:

\begin{prop}\label{howe-local-character}
Let $(\pi,V)$ be an irreducible admissible representation of $G$. Then, there exists a neighborhood $U$ of $0\in\fg$ such that for regular $X\in U$,
\[
Ch_\pi(\exp(X))=\sum_{\cO}c_\cO(\pi)\widehat{\mu}_\cO(X),
\]
or, equivalently, for a function $f$ supported on $\exp(U)$,
\[
\mathrm{tr}\,\pi(f)=\sum_\cO c_\cO(\pi)\mu_\cO(\widehat{f\circ\exp}),
\]
where the sum runs over nilpotent orbits of $\fg$, and the constant $c_\cO(\pi)\in\C$ only depends on $\pi$ and $\cO$.
\end{prop}

\begin{lemma}\label{induced-char}
Let $\cO$ be a nilpotent orbit in $\GL_n(F)$ or $\GL_m(D)$, and let $P$ be the corresponding parabolic subgroup. Then, for $X\in\fg$ regular in a sufficiently small neighborhood of $0$,
\[
\widehat\mu_\cO(X)=Ch_{\Ind_P^G(1_P)}(\exp X).
\]
In particular, for $f\in C_c^\infty(\fg)$ supported on a sufficiently small neighborhood of $0$,
\[
\mu_\cO(\widehat f)=\mathrm{tr} \Ind_P^G(1_P)(f).
\]
\end{lemma}
\begin{proof}
For $\GL_n(F)$, this is \cite[Lem~5]{howe}, and the same argument works for $\GL_m(D)$.
\end{proof}

Let $(\pi,V)$ be an irreducible admissible representation of $\GL_n(F)$. Then, we may express $\dim(\pi^{K_N})$ as a polynomial of $q^{N-1}$, as in \cite{savin}:
\begin{cor}\label{fixed-poly} Let $\pi$ be an irreducible admissible representation of $\GL_n(F)$. Then,
\[
G_\pi(X)=\sum_{\cO\in\cN_G} c_\cO(\pi) \begin{bmatrix}n\\n_1,\dots,n_r\end{bmatrix}_q X^{\frac12\dim\cO},
\]
where $n=n_1+\cdots+n_r$ is the partition corresponding to $\cO$ under the identification in Example~\ref{gln-nilp}.
\end{cor}
\begin{proof}
Let $V$ be the underlying vector space for $\pi$. Let $e_{K_N}$ be the normalized characteristic function, i.e.,
\[
e_{K_N}(x)=\begin{cases}\frac1{\mathrm{vol}(K_N)}&\text{if }x\in K_N\\0&\text{if }x\notin K_N.\end{cases}
\]
Then, $\pi(e_{K_N})\colon\pi\to\pi$ is a projection to $\pi^{K_N}$, so
\[
\dim(\pi^{K_N})=\mathrm{tr}\,\pi(e_{K_N}).
\]
Now by Proposition~\ref{howe-local-character} and Lemma~\ref{induced-char}, for any integer $N\gg0$,
\begin{align*}
\mathrm{tr}\,\pi(e_{K_N})&=\frac1{\mathrm{vol}(K_N)}\int_{K_N}Ch_\pi(g)dg\\
&=\sum_\cO c_\cO(\pi)\frac1{\mathrm{vol}(K_N)}\int_{K_N}Ch_{\Ind_P^G(1)}(g)dg\\
&=\sum_\cO c_\cO\mathrm{tr}\,(\Ind_P^G(1_P))(e_{K_N}).
\end{align*}
Now, by Lemma~\ref{induced-dim} and Example~\ref{gln-nilp},
\begin{align*}
\mathrm{tr}\,(\Ind_P^G(1_P))(e_{K_N})&=\dim(\Ind_P^G(1_P)^{K_N})\\
&=\begin{bmatrix}n\\n_1,\dots,n_r\end{bmatrix}_qq^{\frac12\dim\cO(N-1)}.\qedhere
\end{align*}
\end{proof}

\begin{rmk}
More generally, let $(\pi,V)$ be an admissible irreducible representation of $\GL_n(F)$, let $N\gg0$, and let $\varphi$ be an arbitrary irreducible representation of $K_N$. Then, by \cite[Prop~4.4]{Bush-Henn}, 
\[
e_\varphi(x)\colonequals\begin{cases}\frac{\dim\varphi}{\mathrm{vol}(K_N)}\mathrm{tr}\, (\varphi(x^{-1}))&\text{if }x\in K_N\\
0&\text{if }x\notin K_N\end{cases}
\]
is such that $\pi(e_\varphi)\colon V\to V$ is a projection to the $\varphi$-isotypic component, so a similar argument shows
\[
\dim(\hom_{K_N}(\varphi,\pi|_{K_N}))=\sum_{\cO\in\cN_G} c_\cO(\pi) \dim(\varphi^{K_N\cap P})\begin{bmatrix}n\\n_1,\dots,n_r\end{bmatrix}_qq^{\frac12\dim\cO(N-1)}.
\]
Thus, the coefficients $c_\cO$ completely determine the structure of the $K_N$-representation $\pi|_{K_N}$.
\end{rmk}

It is known that:
\begin{prop}\label{cusp-asymp} For $\rho$ a supercuspidal representation of $\GL_n(F)$,
\[
G_\rho(X)=[n!]_qX^{\frac{n(n-1)}2}+(\text{lower order terms}).
\]
\end{prop}
\begin{proof}
Follows from \cite[Thm~3]{howe}, together with Corollary~\ref{fixed-poly}.
\end{proof}

Moreover, in for certain supercuspidal representations, the entire growth polynomial is known \cite[Prop~5.3]{murnaghan}:

\begin{prop}\label{murn}
Let $\rho$ be the supercuspidal representation corresponding to the generic character $(E/F,\theta)$ where $\theta=(\chi\circ N_{E/F})\phi$ for some generic character $\phi$ of $E$ over $F$, where $\phi$ has level $j$ (i.e., $j$ is the minimal integer such that $\phi$ is trivial on $1+\p_E^j$). Then,
\begin{enumerate}
    \item If $E/F$ is unramified,
    \[
    G_\rho(X)=\sum_\cO(-1)^{n+r}n(r-1)!w_\cO^{-1}q^{(n^2-n-\dim\cO)j/2}\begin{bmatrix}n\\n_1,\dots,n_r\end{bmatrix}_qX^{\frac12\dim\cO}
    \]
    \item If $E/F$ is totally ramified,
    \[
    G_\rho(X)=\sum_\cO(-1)^{n+r}(r-1)!w_\cO^{-1}\frac q{q-1}(r-\sum_{i=1}^rq^{-n_i})q^{\frac{(n^2-n-\dim\cO)j}{2n}}\begin{bmatrix}n\\n_1,\dots,n_r\end{bmatrix}_qX^{\frac12\dim\cO},
    \]
\end{enumerate}
where $w_\cO$ is the number of permutations of $r$ letters which fix the $r$-tuple $(n_1,\dots,n_r)$.
\end{prop}
\begin{proof}
Follows from \cite[Prop~5.3]{murnaghan}, together with Corollary~\ref{fixed-poly}.
\end{proof}

\begin{rmk}
For $\GL_\ell(F)$ with $\ell$ prime, this result suffices to give an explicit formula for all supercuspidal representations.
\end{rmk}

\section{Under Langlands functoriality}\label{section-functoriality}
Again, throughout this section, \emph{assume $F$ has characteristic zero}.


\subsection{Jacquet-Langlands correspondence}


Recall the Jacquet-Langlands correspondence \cite{JL}:
\begin{prop}\label{jl}
Let $D$ be a division algebra of dimension $d^2$ over $F$, and let $E^2(\GL_n(F))$ (resp., $E^2(\GL_m(D))$) be the set of irreducible square-integrable representations of $\GL_n(F)$ (resp, $\GL_m(D)$), where $n=md$. Then, there exists a bijection
\begin{align*}
E^2(\GL_n(F))&\to E^2(\GL_m(D))\\
\pi&\mapsto \pi_D
\end{align*}
such that for $g'\in\GL_m(D)$ a regular element, whose conjugacy class corresponds to a regular element $g\in\GL_n(F)$, then
\[
(-1)^mCh_{\pi_D}(g')=(-1)^nCh_\pi(g).
\]
\end{prop}

Now, we have:
\begin{prop}\label{j-l-formula}
Let $\pi\in E^2(\GL_n(F))$ be an irreducible square-integrable representation, corresponding to $\pi_D\in E^2(\GL_m(D))$ under the Jacquet-Langlands correspondence. Then
\[
(-1)^mc_\cO(\pi)=(-1)^nc_{\cO'}(\pi_D),
\]
where $\cO'\subseteq\cO$ are nilpotent orbits of $G'$ and $G$, respectively (via identifying the Lie algebras of $G'$ and $G$, tensored with $\overline F$).
\end{prop}
\begin{proof}
Let $m=m_1+\dots+m_r$ be an arbitrary partition of $m$, and consider a block diagonal matrix $g=(g_1,\dots,g_r)\in\GL_n(F)$ where $g_i\in\GL_{dm_i}(F)$ are elliptic elements (i.e., has an irreducible characteristic polynomial) and the minimal polynomials of $g_i$ are coprime. Then $g$ is regular semisimple, corresponding to the diagonal matrix $g'=(g_1',\dots,g_r')\in\GL_m(D)$ where each $g_i'\in D_{dm_i}^\times\subset\GL_{m_i}(D)$ is similar to $g_i\in\GL_{dm_i}(F)$, where $D_{dm_i}\supset D$ is a division algebra of dimension $d^2m_i^2$ over $F$.

Then, by Proposition~\ref{howe-local-character} and Lemma~\ref{induced-char}, when $g$ as above is sufficiently close to $1$,
\[
Ch_\pi(g)=\sum_{\cO}c_\cO(\pi)Ch_{\Ind_P^G(1)}(g).
\]
Here by \cite[Thm~3]{dijk} the character $Ch_{\Ind_P^G(1)}(g)$ is locally supported on the union of the conjugates of $P$, and so $Ch_{\Ind_P^G(1)}(g)=0$ for $P$ not corresponding to a partition of the form $a_1d+\cdots+a_id$.

Let $\cN_G^d$ be the set of such nilpotent orbits, i.e., those corresponding to partitions with all parts a multiple of $d$. There is a bijection $\cN_{G'}\simeq\cN_G^d$, sending $\cO'\in\cN_{G'}$ to the unique nilpotent orbit $\cO$ of $G$ such that $\cO'\otimes\overline F\subseteq\cO\otimes\overline F$ (on the level of partitions, $a_1+\cdots+a_i\mapsto a_1d+\cdots+a_id$).

Now, Proposition~\ref{jl} gives the equality
\begin{equation}\label{jl-eq}
(-1)^m\sum_{\cO\in\mathcal N_G^d}c_\cO(\pi)Ch_{\Ind_P^G(1)}(g)=(-1)^n\sum_{\cO'\in\mathcal N_{G'}}c_{\cO'}(\pi_D)Ch_{\Ind_P^G(1)}(g),
\end{equation}
since by Lemma~\ref{induced-char},
\[
\widehat\mu_{\cO'}(g')=Ch_{\Ind_{P'}^{G'}(1_{P'})}(g')=Ch_{\Ind_P^G(1_P)}(g).
\]
where $P'\colonequals (P\otimes\overline F)\cap G'\subset G'$. Moreover, for parabolic subgroups $P_1\subsetneq P_2$ of $G$ we have $\supp(Ch_{\Ind_{P_1}^G(1)})\subsetneq\supp(Ch_{\Ind_{P_2}^G(1)})$ in a neighborhood of $1\in G$, so the characters $\{Ch_{\Ind_P^G(1)}:P\subset G\text{ parabolic}\}$ are linearly independent. Thus, equation~\ref{jl-eq} gives \[(-1)^mc_\cO(\pi)=(-1)^nc_{\cO'}(\pi_D).\qedhere\]
\end{proof}

\begin{cor}\label{const-term}
Let $\pi\in E^2(\GL_n(F))$ be an irreducible square-integrable representation, and let $\pi_D\in E^2(D^\times)$ be the corresponding representation under the Jacquet-Langlands correspondence, where $D$ is a $n^2$-dimensional division algebra. Then the growth polynomial $G_\pi(X)$ has constant term $(-1)^{n-1}\dim(\pi_D)$.
\end{cor}
\begin{proof}
Since $D^\times$ is compact modulo center, the representation $\pi_D$ is finite-dimensional. Thus, $c_0(\pi_D)=\dim(\pi_D)$, and so $c_0(\pi)=(-1)^{n-1}c_0(\pi_D)=(-1)^{n-1}\dim(\pi_D)$.
\end{proof}

The combination of Proposition~\ref{cusp-asymp} and Corollary~\ref{const-term} allows the calculation for the polynomial $\dim(\pi^{K_N})$ for representations $\pi$ of $\GL_2(F)$:

\begin{example}\label{gl2-dim}
When $n=2$, \cite[Section~56]{Bush-Henn} provides an explicit description of the Jacquet-Langlands correspondence for minimal supercuspidal representations $\rho$ of $\GL_2$. Only treating these suffices since arbitrary supercuspidal representations are twists of minimal supercuspidal representations.
\begin{itemize}
    \item If $\rho$ has level zero, then $\rho_D$ is induced from a character of $E^\times U_D^1$, so
    \[
    G_\rho(X)=(q+1)X-2.
    \]
\end{itemize}
Otherwise, there is a simple stratum $(\fA, n,\alpha)$ such that $\rho$ contains the character $\psi_\alpha$ of $U_\fA^{[n/2]+1}$, so $\rho\cong\cInd_J^G\Lambda$ where $J=E^\times U_\fA^{[(n+1)/2]}$ and $\Lambda$ is an irreducible representation on $J$:
\begin{itemize}
    \item If $e_\fA=2$ then $\rho_D$ is induced from a character on $E^\times U_D^{[n/2]+1}$, so
    \[
    G_\rho(X)=(q+1)X-|E^\times U_D^{[n/2]+1}\backslash D^\times|.
    \]
    \item If $e_\fA=1$ and $n$ is odd then $\rho_D$ is induced from a $q$-dimensional representation of $E^\times U_D^n$, so
    \[
    G_\rho(X)=(q+1)X-q|E^\times U_D^n\backslash D^\times|,
    \]
    \item If $e_\fA=1$ and $n$ is even then $\rho_D$ is induced from a character of $E^\times U_D^{n+1}$, so
    \[
    G_\rho(X)=(q+1)X-|E^\times U_D^{n+1}\backslash D^\times|.
    \]
\end{itemize}
In each of these cases, we have:
\begin{align*}
|E^\times U_D^m\backslash D^\times|&=|\langle\varpi_E\rangle\backslash\langle\varpi_D\rangle|\cdot|k_E^\times\backslash k_D^\times|\cdot|U_E^1U_D^m\backslash U_D^1|\\
&=\frac2e\cdot \frac{q^2-1}{q^{2/e}-1}|U_E^1U_D^m\backslash U_D^1|.
\end{align*}
Now, there is an exact sequence
\[
1\to (U_E^1\cap U_D^m)\backslash U_E^1\cong U_D^m\backslash U_E^1U_D^m\to U_D^m\backslash U_D^1\to U_E^1U_D^m\backslash U_D^1\to 1,
\]
where $m\ge1$. Noting that $U_E^1\cap U_D^m=U_E^{1+[(m-1)e/2]}$, we obtain
\[
|U_E^1U_D^m\backslash U_D^1|=\frac{|U_D^m\backslash U_D^1|}{|U_E^{1+[(m-1)e/2]}\backslash U_E^1|}=q^{2(m-1)-\frac2e[\frac e2(m-1)]}.\]
Thus, combining everything gives
\[
G_\rho(X)=\begin{cases}
(q+1)X-2&\text{if }\rho\text{ has level zero}\\
(q+1)X-(q+1)q^{[n/2]}&\text{if }e_\fA=2\\
(q+1)X-2q^n&\text{if }e_\fA=1.
\end{cases}
\]
\end{example}

\subsection{Base change} Another instance of Langlands functoriality is base change and automorphic induction. For a finite extension $E/F$ of degree $d$, base change is given by the commutative diagram
\begin{equation}\label{eq-base-change} \begin{tikzcd}
\mathbf{Irr}(\GL_n(F)) \arrow{r}{\mathrm{rec}_F} \arrow[swap,dashed]{d}{\mathrm{BC}_{E/F}} & \mathcal G_n(F)  \arrow{d}{\mathrm{Res}^{\mathcal W_F}_{\mathcal W_E}} \\%
\mathbf{Irr}(\GL_n(E)) \arrow{r}{\mathrm{rec}_E}& \mathcal G_n(E),
\end{tikzcd}\end{equation}
and automorphic induction is given by
\[ \begin{tikzcd}
\mathbf{Irr}(\GL_n(E)) \arrow{r}{\mathrm{rec}_E} \arrow[swap,dashed]{d}{\mathrm{AI}_{E/F}} & \mathcal G_n(E)  \arrow{d}{\mathrm{Ind}^{\mathcal W_F}_{\mathcal W_E}} \\%
\mathbf{Irr}(\GL_{nd}(F)) \arrow{r}{\mathrm{rec}_F}& \mathcal G_{nd}(F),
\end{tikzcd}\]
where $\mathrm{rec}_F$ denotes the local Langlands correspondence. When $E/F$ is a cyclic extension of prime degree, $\mathrm{BC}_{E/F}$ has a characterization in terms of characters \cite[Thm~6.2]{AC}:
\begin{defn}
Let $E/F$ be a cyclic extension of prime degree, and let $\tau\in\Gal(E/F)$ be a generator.
Moreover, let $(\Pi,V)$ be an irreducible representation of $\GL_n(E)$ such that $\Pi\cong\Pi^\tau$, where $\Pi^\tau(g)\colonequals\Pi(\tau(g))$ for $g\in\GL_n(E)$. Let $I_\tau\colon\Pi\to\Pi^\tau$ be such an isomorphism, appropriately normalized (require the Whittaker functional to be fixed). Then there exists a locally constant function $Ch_{\Pi,\tau}$ defined on $\GL_n(E)^{\tau-reg}$ such that
\[
\mathrm{tr}(\pi(f)\circ I_\tau)=\int_{\GL_n(E)^{\tau-reg}}f(g)Ch_{\Pi,\tau}(g)dg,
\]
where $\GL_n(E)^{\tau-reg}$ is the set of $g\in\GL_n(E)$ such that $g\tau(g)\cdots\tau^{d-1}(g)$ is regular.
\end{defn}

\begin{prop}[Shintani character formula]\label{shintani-char}
Let $E/F$ be a cyclic extension of prime degree $d$, and let $\tau\in\Gal(E/F)$ be a generator. Then for $\pi$ an irreducible tempered representation of $\GL_n(F)$,
\[
Ch_{\mathrm{BC}_{E/F}(\pi),\tau}(g)=Ch_\pi(\mathcal Ng),
\]
for any $g\in\GL_n(E)^{\tau-reg}$, where $\mathcal Ng\in\GL_n(F)$ is the norm of $g$, i.e., the unique (up to conjugacy) element of $\GL_n(F)$ conjugate to $g\tau(g)\cdots\tau^{d-1}(g)$ (which is regular by definition).
\end{prop}

Shintani's character formula gives another method to calculate the growth polynomial:
\begin{prop}\label{twisted-char}
Let $E/F$ be a cyclic extension of prime degree $d$, and let $\theta$ be a generic character of $E/F$ (i.e., $\theta\ne\theta^\tau$). Then for an integer $N\gg0$,
\[
G_{\mathrm{AI}_{E/F}\theta}(q^{[N/e]-1})=\mathrm{tr}(I_\tau|_{V^{K_N}})
\]
where $V\colonequals\theta\times\theta^\tau\times \cdots\times\theta^{\tau^{d-1}}$.
\end{prop}
\begin{proof}
Let $\pi=\mathrm{AI}_{E/F}(\theta)$ and $\Pi=\mathrm{BC}_{E/F}(\pi)$. Then, by Mackey's formula,
\begin{align*}
    \mathrm{rec}_E(\Pi)&=(\Ind_{\mathcal W_E}^{\mathcal W_F}\mathrm{rec}_E\theta)|_{\mathcal W_E}\\
    &=\bigoplus_{\sigma\in\mathcal W_E/\mathcal W_F}\mathrm{rec}_E\theta^\sigma\\
    &=\mathrm{rec}_E\theta\oplus\mathrm{rec}_E\theta^\tau\oplus \cdots\oplus\mathrm{rec}_E\theta^{\tau^{d-1}},
\end{align*}
so that $\Pi=\theta\times\cdots\times\theta^{\tau^{d-1}}$. Now, for $N\gg0$, by Proposition~\ref{shintani-char} and Lemma~\ref{induced-char},
\begin{align*}
    \mathrm{tr}(I_\tau|_{\Pi^{K_N}})&=\frac1{\mathrm{vol}(K_N(E))}\int_{K_N(E)^{\tau-reg}}Ch_{\Pi,\tau}(g)dg\\
    &=\frac1{\mathrm{vol}(K_N(E))}\int_{K_N(E)^{\tau-reg}}Ch_{\pi}(\mathcal Ng)dg\\
    &=\sum_\cO c_\cO(\pi)\frac1{\mathrm{vol}(K_N(E))}\int_{K_N(E)^{\tau-reg}}Ch_{\Ind_P^G(1)}(\mathcal Ng)dg,
\end{align*}
where the use of Lemma~\ref{induced-char} is justified because $\mathcal Ng\in\GL_n(F)$ is regular and lies in a small neighborhood of $1$.
We may safely replace $\Ind_P^G(1)$ with $\Ind_P^G(\chi_1\boxtimes\cdots\boxtimes\chi_r)$ for unitary characters $\chi_i$ such that $\chi_i\circ N_{E/F}$ are distinct, since the germs of their characters match. But then (by following the diagram~\eqref{eq-base-change}), the base change is \[\mathrm{BC}_{E/F}(\Ind_P^G(\chi_1\boxtimes\cdots\boxtimes\chi_r))=\Ind_{P(E)}^{G(E)}(\chi_1\circ N_{E/F}\boxtimes\cdots \boxtimes\chi_r\circ N_{E/F}),\] so by the same argument as above,
\begin{align*}
\frac1{\mathrm{vol}(K_N(E))}\int_{K_N(E)^{\tau-reg}}Ch_{\Ind_P^G(1)}(\mathcal Ng)dg&=\mathrm{tr}\big(I_\tau|_{\Ind_{P(E)}^{G(E)}(\chi_1\circ N_{E/F}\boxtimes\cdots\boxtimes \chi_r\circ N_{E/F})^{K_N}}\big).
\end{align*}
Vectors of the invariant space $\Ind_{P(E)}^{G(E)}(\chi_1\circ N_{E/F}\boxtimes\cdots \boxtimes\chi_r\circ N_{E/F})^{K_N}$ can be viewed as functions on $P(E)\backslash G(E)/K_N(E)$, and $\tau$ acts on it by the Galois action, so the trace is the size of the invariant elements
\[
(P(E)\backslash G(E)/K_N(E))^\tau=P(F)\backslash G(F)/(K_N(E)\cap G(F))=P(F)\backslash G(F)/K_{[N/e]}(F),
\]
so altogether, by Corollary~\ref{fixed-poly},
\begin{align*}
\mathrm{tr}(I_\tau|_{\Pi^{K_N}})&=\sum_\cO c_\cO(\pi)|P(F)\backslash G(F)/K_{[N/e]}(F)|\\
&=G_{\mathrm{AI}_{E/F}\theta}(q^{[N/e]-1}).\qedhere
\end{align*}
\end{proof}

We will give an alternative proof of the $n=2$ case of Proposition~\ref{murn}~(1), as well as Example~\ref{gl2-dim}:
\begin{prop}
Let $E/F$ be an unramified degree $2$ extension and let $\theta$ be a generic character of $E/F$, such that $\theta^\tau\theta^{-1}$ has level $\ell$ (i.e., $\ell$ is the minimal integer such that $\theta^\tau\theta^{-1}$ is trivial on $1+\p_E^{\ell+1}$). Then,
\[
G_{\mathrm{AI}_{E/F}(\theta)}(X)=(q+1)X-2q^\ell.
\]
\end{prop}
\begin{proof}
Let $E=F(\sqrt\alpha)$ with $v_F(\alpha)=0$ and let $\varpi\in\sO_F$ be the generator of the maximal ideal $\p_F$. By Proposition~\ref{twisted-char}, it suffices to calculate the trace of $I_\tau$ on $(\theta\times\theta^\tau)^{K_N(E)}$. By \cite[Prop~4.5.6]{bump} it is given by
\begin{align*}
    q^{-\ell}I_\tau\colon\nInd_{B(E)}^{G(E)}(\theta\boxtimes\theta^\tau)&\to \nInd_{B(E)}^{G(E)}(\theta\boxtimes\theta^\tau)^\tau\\
    f&\mapsto\bigg(g\mapsto\lim_{s\to0^-} \int_Ef\big(\begin{pmatrix}&1\\1&x\end{pmatrix}\tau(g)\big)|x|^sdx\bigg).
\end{align*}
For $g_0\in\GL_2(E)$, define the twisted indicator function
\[
\one_{g_0}(g)\colonequals\begin{cases}(\theta(\frac12)\boxtimes\theta^\tau(-\frac12))(b)&\text{if }g=bg_0u\text{ for some }b\in B(E),u\in K_N(E)\\
0&\text{if }g\notin B(E)g_0K_N(E),\end{cases}
\]
so that $\{\one_{g_0}\}_{g_0\in B(E)\backslash\GL_2(E)/K_N(E)}$ is a basis for $(\theta\times\theta^\tau)^{K_N(E)}$. Note that by Iwasawa decomposition we may take $g_0\in K_0(E)$, in which case $\one_{g_0}$ is well-defined for $N>\ell$. Thus,
\begin{equation}\label{trace-integral}
\mathrm{tr}(q^{-\ell}I_\tau|_{(\theta\times\theta^\tau)^{K_N}})=\sum_{g_0\in B(E)\backslash\GL_2(E)/K_N(E)}\lim_{s\to 0^-}\int_E\one_{g_0}\big(\begin{pmatrix}&1\\1&x\end{pmatrix}\tau(g_0)\big)|x|^sdx.
\end{equation}
Moreover, for arbitrary $k\in K_0(F)$ we have
\[
\one_{g_0}\big(\begin{pmatrix}&1\\1&x\end{pmatrix}\tau(g_0)\big)=\one_{g_0k}\big(\begin{pmatrix}&1\\1&x\end{pmatrix}\tau(g_0k)\big),
\]
since if $\begin{pmatrix}&1\\1&x\end{pmatrix}\tau(g_0)=bg_0u$ then $\begin{pmatrix}&1\\1&x\end{pmatrix}\tau(g_0k)=bg_0k(k^{-1}uk)$, where $k^{-1}uk\in K_N(E)$. Thus, the sum in equation~\eqref{trace-integral} can be combined by $K_0(F)$-orbits. Those orbits are represented by
\[
\big\{I_2,\begin{pmatrix}1\\\varpi^i\sqrt{\alpha}&1\end{pmatrix}:0\le i<N\big\},
\]
where the orbit of $I_2$ has size $q^{N-1}(q+1)$ by Lemma~\ref{parabolic-size}, and the orbit of $\begin{pmatrix}1\\\varpi^i\sqrt{\alpha}&1\end{pmatrix}$ has size $(q^2-q)q^{2(N-1)}$ for $i=0$ and $(q^2-1)q^{2(N-1)-i}$ for $i>0$. What is left is to calculate the integral in equation~\ref{trace-integral} for each of these choices of $g_0$.

When $g_0=I_2$, then $\begin{pmatrix}&1\\1&x\end{pmatrix}\in BK_N$ exactly when $v_E(x)\le -N$, in which case $b=\begin{pmatrix}-x^{-1}&*\\&x\end{pmatrix}$ and so the desired integral is
\[
\lim_{s\to 0^-}\int_{v(x)\le -N}\theta(-x^{-1})\theta^\tau(x)|x|^sdx.
\]
The change of variables $x\mapsto ux$ for $u\in\sO_E^\times$ such that $\theta(\tau(u))\ne\theta(u)$ shows the integral is $0$.

Next, when $g_0=\begin{pmatrix}1\\\varpi^i\sqrt{\alpha}&1\end{pmatrix}$, we have
\[
\begin{pmatrix}&1\\1&x\end{pmatrix}\begin{pmatrix}1\\-\varpi^i\sqrt{\alpha}&1\end{pmatrix}=\begin{pmatrix}*&*\\1-x\varpi^i\sqrt\alpha&x\end{pmatrix}\in Bg_0K_N
\]
exactly when $(1-x\varpi^i\alpha)x^{-1}\equiv\varpi^i\alpha\pmod{\p_E^N}$, i.e., $x\in\frac1{2\varpi^i\sqrt\alpha}(1+\p_E^{N-i})$. Thus, the integral is (using that the measure $\frac{dx}{|x|}$ is a Haar measure for $E^\times$):
\[
\int_{\frac1{2\varpi^i\sqrt\alpha}(1+\p_E^{N-i})}\theta(-x^{-1})\theta^\tau(x)\frac{dx}{|x|}=\int_{1+\p_E^{N-i}}\theta(x^{-1})\theta^\tau(x)\frac{dx}{|x|}.
\]
If $N-i\ge\ell+1$ then $\theta(x^{-1})\theta^\tau(x)=1$ on $1+\p_E^{N-i}$, so the integral is simply $\int_{1+\p^{N-i}}\frac{dx}{|x|}=q_E^{(i-N)+\frac12}=q^{2(i-N)+1}$. On the other hand, if $N-i\le\ell$ then looking at the change of variables $x\mapsto ux$ for $u\in 1+\p_E^\ell$ such that $\theta(\tau(u))\ne\theta(u)$ shows the integral is $0$. Thus, in conclusion,
\begin{align*}
    q^{-\ell}\mathrm{tr}(I_\tau|_{(\theta\times\theta^\tau)^{K_N}})&=(q^2-q)q^{2(N-1)}q^{-2N+1}+\sum_{i=0}^{N-\ell-1}(q^2-1)q^{2(N-1)-i}q^{2(i-N)+1}\\
    &=(q+1)q^{N-\ell-1}-2.\qedhere
\end{align*}
\end{proof}

\begin{rmk}
More generally, let $\rho\in\mathbf{Irr}(\GL_n(E))$ be supercuspidal and consider $\pi=\mathrm{AI}_{E/F}\rho$. Then, the same arguments show the growth polynomial $G_\pi$ is related to the trace of the intertwining integral $\rho\times\cdots\times\rho^{\tau^{d-1}}\to(\rho\times\cdots\times\rho^{\tau^{d-1}})^\tau$. Moreover, the trace is given by a sum over the double coset $P_{n+\cdots+n}(E)\backslash\GL_{nd}(E)/K_N(E)K_0(F)$, the $K_0(F)$-orbits in the partial flag variety over $\sO_E/\p_E^N$. Looking more closely at the combinatorial structure of such double cosets is expected to give new results.
\end{rmk}

\section*{Acknowledgement}
The author thanks Hao Peng for mentoring this project and giving helpful advice. We thank Zhiyu Zhang for suggesting this project. We thank Prof. Wei Zhang and Prof. Ju-Lee Kim for giving us some relevant references and pointing us in the right direction. We thank Dr. Max Gurevich for pointing out an error in the original argument for Proposition~\ref{delta-dim}. We thank Prof. Marie-France Vign\'{e}ras and Prof. Guy Henniart for helpful discussions, as well as fixing an error in Remark~\ref{gl2-dim}.

Finally, we thank Prof. David Jerison and Prof. Ankur Moitra for organizing the SPUR program and helping us through insightful discussions.



\appendix
\section{Representations of $\SL_n(F)$}\label{section-sl-n}

Let $(\varphi,V)$ be an irreducible representation of $\SL_n(F)$. We reduce the calculation of the \emph{leading term} of the growth polynomial of $\varphi$ to the theory for $\GL_n(F)$.

We first extend $\varphi$ to a representation of $Z\SL_n(F)$, which is of finite-index in $\GL_n(F)$. The central character
\[
\omega_\varphi\colon Z\cap\SL_n(F)\to\C^\times
\]
can be extended to a character $\omega\colon Z\to\C^\times$, since the Pontryagin dual of $Z\cap\SL_n(F)\hookrightarrow Z$ is surjective. Now, consider an extension of $\varphi$
\begin{align*}
\widetilde\varphi\colon Z\SL_n(F)&\to\GL(V)\\
zg&\mapsto\omega(z)\varphi(g).
\end{align*}
Since $Z\SL_n\subset\GL_n$ has finite index, by \cite[Lem~2.7]{Bush-Henn} the induced representation $\Ind_{Z\SL_n}^{\GL_n}\widetilde\rho$ decomposes into a direct product of irreducible representations of $\GL_n(F)$, say as $\pi_1\oplus\cdots\oplus\pi_k$. Moreover, by \cite[Thm~1.2]{tadic} the restrictions $\pi_i|_{Z\SL_n(F)}$ only contain $\widetilde\varphi$ once, so the $\pi_i$ are distinct:
\[
\hom_{\GL_n}(\pi_i,\Ind_{Z\SL_n}^{\GL_n}\widetilde\varphi)=\hom_{Z\SL_n}(\pi_i|_{Z\SL_n},\widetilde\varphi)
\]
is one-dimensional. Since there are nonzero homomorphisms
\[
\pi_i|_{Z\SL_n}\to\widetilde\varphi\to\pi_j|_{Z\SL_n},
\]
by \cite[Prop~2.4]{tadic} there exists a character $\chi\colon\GL_n/Z\SL_n\to\C^\times$ such that $\pi_i\cong\chi\pi_j$. Conversely, if $\pi_i$ is an irreducible component of $\Ind_{Z\SL_n}^{\GL_n}\widetilde\varphi$ then so is any twist $\chi\pi_1$ by a character $\chi\colon\GL_n/Z\SL_n\to\C^\times$. Thus, in fact
\[
(\Ind_{Z\SL_n}^{\GL_n}\widetilde\varphi)^{\oplus d}\cong\bigoplus_{\chi\colon \GL_n/Z\SL_n\to\C^\times}\chi\pi_1,
\]
where $d$ is the size of the subgroup
\[
\{\chi\colon \GL_n/Z\SL_n\to\C^\times:\chi\pi_1\cong\pi_1\}.
\]
Note that $d$ is finite since $\GL_n/Z\SL_n\cong F^\times/(F^\times)^n\cong\mathbb Z/n$ is finite. Now, the $K_N$ invariants are, by Lemma~\ref{fixed},
\begin{equation}\label{sl-n-eq}
\bigoplus_{g\in Z\SL_n\backslash\GL_n/K_N}\widetilde\varphi^{Z\SL_n\cap gK_Ng^{-1}}=\bigoplus_{g\in Z\SL_n\backslash\GL_n/K_N}\widetilde\varphi^{gK_Ng^{-1}}\cong (\pi_1^{K_N})^{\oplus n/d}.
\end{equation}
Since each $\{gK_Ng^{-1}\}$ is a fundamental system of neighborhoods of $\SL_n$ decreasing at the same rate, by Harish-Chandra's local character formula they all behave the same asymptotically:
\[
\dim(\widetilde\varphi^{gK_Ng^{-1}})=(1+o(1))\dim(\widetilde\varphi^{hK_Nh^{-1}}).
\]
Thus, equation~\ref{sl-n-eq} gives:
\[
\frac nd\dim(\pi_1^{K_N})=(1+o(1))\cdot |Z\SL_n\backslash\GL_n/K_N|\dim(\widetilde\varphi^{K_N})=(n+o(1))\dim(\varphi^{K_N\cap\SL_n}),
\]
so that
\[
\dim(\varphi^{K_N\cap\SL_n})=(\frac1d+o(1))\dim(\pi_1^{K_N}).
\]

To summarize:

\begin{prop}\label{sl-n-dim}
Let $\varphi$ be an irreducible representation of $\SL_n$. There exists an irreducible representation $\pi_1$ of $\GL_n$ such that $\varphi\hookrightarrow\pi_1|_{\SL_n}$. Then,
\[
\dim(\varphi^{K_N\cap\SL_n})=\big(\frac 1d+o(1)\big)\dim(\pi_1^{K_N}),\]
where $d$ is the number of characters $\chi\colon\GL_n/Z\SL_n(F)\to\C^\times$ such that $\chi\pi_1\cong\pi_1$ as $\GL_n$-representations.
\end{prop}

\begin{rmk}
If $\varphi$ is supercuspidal, then so is $\pi_1$ \cite{tadic}. In the language of \cite{Bushnell}, let $\pi_1=\cInd_{\mathbf J}^G\Lambda$, where $(\mathbf J=\mathbf J(\beta,\mathfrak a),\Lambda)$ is an EMST in $G$. Then, by \cite[Lem~3.1.3]{Bushnell} and \cite[3.1.4]{Bushnell}, $d=e(F[\beta]/F)$, the ramification degree, where $\beta\in\GL_n(F)$ is such that $F[\beta]$ is a field.
\end{rmk}

The correcting factor $d$ can be deduced from the Bernstein-Zelevinsky classification of the representation $\pi_1$:
\begin{cor}
Let $\varphi\hookrightarrow\pi_1|_{\SL_n(F)}$ be an irreducible representation of $\SL_n$, where $\pi_1=\langle a\rangle$, with $a=\{\Delta_1,\dots,\Delta_m\}$ a multisegment, where $\Delta_i=[\rho_i,\nu^{r_i-1}\rho_i]$ for a supercuspidal representation $\rho_i$ of $\GL_{n_i}(F)$. Let $d$ be the number of characters $\chi\colon F^\times/(F^\times)^n\to\C^\times$ such that $\{\Delta_1\otimes\chi,\dots,\Delta_m\otimes\chi\}=\{\Delta_1,\dots,\Delta_m\}$ as multisets, where $[\rho,\nu^{r-1}\rho]\otimes\chi\colonequals[\rho\otimes\chi\circ\det,\nu^{r-1}\rho\otimes\chi\circ\det]$. Then,
\[
\dim(\varphi^{K_N\cap\SL_n})=(1+o(1))\frac1d\frac{[n!]_q}{[r_1!]_{q^{n_1}}\cdots[r_m!]_{q^{n_m}}}(q^{N-1})^{\frac12(n^2-n_1r_1^2-\cdots-n_mr_m^2)},
\]
\end{cor}
\begin{proof}
Immediate from Proposition~\ref{sl-n-dim}, Theorem~\ref{main-thm} and the uniqueness of the Bernstein-Zelevinsky classification \cite[Thm~6.1]{BZ2}.
\end{proof}
Thus, as opposed to $\GL_n$, where the leading coefficient of growth polynomials only involved the combinatorial structure of the multisegment, the leading coefficient for the growth polynomial of representations of $\SL_n$ involves the properties of the supercuspidal representations $\rho_i$.
\begin{example}
Let $\varphi\hookrightarrow(\langle[\rho,\nu\rho]\rangle\times\langle[\chi\rho,\nu\chi\rho]\rangle)|_{\SL_n}$ where $\chi(x)\colonequals(-1)^{v_F(x)}$ is a character of $F^\times$, and $\rho=\cInd_{\mathbf J}^{\GL_{n_1}}\Lambda$ where $(\mathbf J=\mathbf J(\beta,\mathfrak a),\Lambda)$ is an EMST in $\GL_{n_1}(F)$, with $d_1=e(F[\beta]/F)$ odd. Then, $2d_1$ characters of $F^\times/(F^\times)^n$ fix the multisegment $\{[\rho,\nu\rho],[\chi\rho,\nu\chi\rho]\}$ (generated by $\chi$ and the characters fixing $\rho$), so:
\[
\dim(\varphi^{K_N\cap\SL_n})=(1+o(1))\frac1{2d_1}\frac{[n!]_q}{[2]_{q^{n_1}}^2}(q^{N-1})^{8n_1^2-4n_1}.
\]
\end{example}

\section{Direct calculation of the Gelfand-Kirillov dimension of supercuspidal representations}\label{appendix-cusp}

We will provide a different proof of a weaker form of Proposition~\ref{cusp-asymp}, based on Bushnell's classification of supercuspidal representations of $\GL_n(F)$:

\begin{thm}[Bushnell-Kutzko {\cite{Bush-Kutz}}]
If $\rho$ is an irreducible supercuspidal representation of $G=\GL_n(F)$, then
\[
\rho\cong \cInd_{\bK}^G\Lambda,
\]
for some compact modulo center subgroup $\bK\supset Z$ and representation $\Lambda$ on $\bK$.
\end{thm}
Here, since $\Lambda$ is an irreducible representation on a compact modulo center subgroup $\bK$, it must be finite-dimensional.

The following lemma on the size of double cosets is useful:

\begin{lemma}\label{coset-size}
For $g\in G$, the size of the double coset $K_0\backslash K_0gK_0/K_N$ depends only $g\in K_0\backslash G/ZK_0$, and when $g=(\varpi^{a_1},\cdots,\varpi^{a_n})$ with $a_1\ge\cdots\ge a_n$ then
\[
|K_0\backslash K_0gK_0/K_N|=q^{\frac12(n_1^2+\cdots+n_r^2-n^2)+\sum_{i<j\le n}\min\{a_i-a_j,N\}}\begin{bmatrix}n\\n_1,\dots,n_r\end{bmatrix}_q,
\]
where $n=n_1+\cdots+n_r$ is a partition such that
\[
a_1=\cdots=a_{n_1}>a_{n_1+1}=\cdots=a_{n_1+n_2}>\cdots>a_{n_1+\dots+n_{r-1}+1}=\cdots=a_n,
\]
and
\[
\begin{bmatrix}n\\n_1,\dots,n_r\end{bmatrix}_q\colonequals \frac{[n!]_q}{[n_1!]_q\cdots[n_r!]_q},
\]
where $[n!]_q\colonequals \prod_{k=1}^n\frac{q^n-1}{q-1}$ is the $q$-analogue of $n!$.
\end{lemma}
\begin{proof}
First of all, there is a bijection
\begin{align*}
    (K_0\cap gK_0g^{-1})\backslash K_0/K_N&\xrightarrow{\sim}K_0\backslash K_0gK_0/K_N\\
    x&\mapsto gx.
\end{align*}
When $g=(\varpi^{a_1},\cdots,\varpi^{a_n})$ with $a_1\ge\cdots\ge a_n\ge0$, we have
\[
K_0\cap gK_0g^{-1}=\begin{pmatrix}\sO^\times&\p^{a_1-a_2}&\cdots\\{*}&\sO^\times&\p^{a_2-a_3}&\\{*}&{*}&\ddots\\
{*}&&&\sO^\times\end{pmatrix}\subset K_0.
\]
Let $\lambda$ be the partition $n=n_1+\cdots+n_r$ and denote the distinct values $a_{n_1+\dots+n_i}$ as $b_i$. Then, if ${^t}P_\lambda$ is the corresponding parabolic subgroup,
\[
{^t}P_\lambda\colonequals \begin{pmatrix}
\GL_{n_1}&0&\cdots&0\\
{*}&\GL_{n_2}&0&\vdots\\
&&\ddots&0\\
{*}&{*}&&\GL_{n_r}
\end{pmatrix},
\]
then $(K_0\cap gK_0g^{-1})/K_N\supset {^t}P_\lambda(\sO/\p^N)$. Now, the abelian group
\[
\frac{1+M_n(\p/\p^N)}{(1+M_n(\p/\p^N))\cap gK_0g^{-1}}
\]
acts simply on the quotient $(K_0\cap gK_0g^{-1})\backslash K_0/K_N$ from the right with quotient $P_\lambda^-(k)\backslash\GL_n(k)$, so
\begin{align*}
|\frac{\GL_n(\sO/\p^N)}{K_0\cap gK_0g^{-1}}|&=|\frac{1+M_n(\p/\p^N)}{(1+M_n(\p/\p^N))\cap gK_0g^{-1}}|\cdot |\frac{\GL_n(k)}{{^t}P_\lambda(k)}|\\
&=\prod_{i<j\le r}|(\p/\p^N)/(\p^{b_i-b_j}/\p^N)|^{n_in_j}\cdot \begin{bmatrix}
n\\
n_1,\dots,n_r
\end{bmatrix}_q\\
&=q^{\sum_{i<j\le r}n_in_j(\min\{b_i-b_j,N\}-1)}\begin{bmatrix}
n\\
n_1,\dots,n_r
\end{bmatrix}_q\\
&=q^{\frac12(n_1^2+\cdots+n_r^2-n^2)+\sum_{i<j\le n}\min\{a_i-a_j,N\}}\begin{bmatrix}
n\\
n_1,\dots,n_r
\end{bmatrix}_q.\qedhere
\end{align*}
\end{proof}

\begin{prop}\label{dim:cuspidal}
Let $\rho$ be an irreducible supercuspidal representation of $\GL_n(F)$. Then,
\[
\dim_{GK}(\rho)=\frac{n(n-1)}2.
\]
\end{prop}

\begin{proof}
Let $\rho=\cInd_{\bK}^G(\lambda)$ for an irreducible representation $\lambda$ of $\bK$. The intersection $K\colonequals \bK\cap\{g\in G:|\det g|=1\}$ must be compact open (since $\cO_F^\times\subset F^\times$ is open), so by conjugating we may assume $K\subset K_0$. Then,
\begin{align*}
    \rho^{K_N}&=\bigoplus_{g\in \bK\backslash G/K_N}\Lambda^{\bK\cap gK_Ng^{-1}}\\
    &=\bigoplus_{g\in\bK\backslash G/K_N}\lambda^{K\cap gK_Ng^{-1}},
\end{align*}
where $\lambda=\Lambda|_K$ is irreducible. Since $\lambda$ is finite-dimensional
\[
\dim(\rho^{K_N})\sim|\bK\backslash S/K_N|,
\]
where $S\colonequals \{g\in G:\lambda^{K\cap gK_Ng^{-1}}\ne0\}$. Moreover, since $\bK\supset ZK\subset ZK_0$ is a chain of finite-index inclusions,
\[
\dim(\rho^{K_N})\sim|ZK_0\backslash S/K_N|.
\]

Here, $\lambda$ is trivial on some compact subgroup $K_m\subset K$.

The general idea of the calculation is to ``bound" the set $S$ by double cosets $ZK_0\backslash G/K_0$. That is, we find subsets $A,B\subset ZK_0\backslash G/K_0$ such that
\[
ZK_0AK_0\subset S\subset ZK_0BK_0,
\]
and calculate the size of the double cosets $ZK_0\backslash ZK_0AK_0/K_N$ and $ZK_0\backslash ZK_0BK_0/K_N$. Such a strategy is fruitful since the double coset $ZK_0\backslash G/K_0$ has a particularly simple description, the Cartan decomposition.

Firstly, since if $K_0\cap gK_Ng^{-1}\subset K_m$ then $\lambda$ is trivial on $K\cap gK_Ng^{-1}\subset  K_m$, we have the inclusion
\[
S\supset \{g\in G:K_0\cap gK_Ng^{-1}\subset K_m\}.
\]
Now, since $ZK_0$ normalizes $K_N$ and $K_m$, the set on the right is a $ZK_0\backslash G/K_0$-double coset. A representative $(\varpi^{a_1},\dots,\varpi^{a_n})\in ZK_0\backslash G/K_0$ is such that
\[
K_0\cap gK_Ng^{-1}=1+(\p^{\max\{0,N+a_i-a_j\}}),
\]
which a subset of $K_m$ exactly when $a_i-a_j<N-m$ for each $i$ and $j$. Thus,
\begin{equation}\label{lower-bound}
S\supset K_0\{(\varpi^{a_1},\cdots,\varpi^{a_n}):a_1\ge\cdots\ge a_n,a_i-a_{i+1}\le N-m\}ZK_0.
\end{equation}

On the other hand,
\begin{equation}\label{upper-bound}
S\subset K_0\{(\varpi^{a_1},\cdots,\varpi^{a_n})\in G:a_1\ge\cdots\ge a_n,\text{ and for each }i,a_i-a_{i+1}\le N\}K_0.
\end{equation}
We will show the converse; let
\[
g=u(\varpi^{a_1},\cdots,\varpi^{a_r},\varpi^{a_{r+1}},\cdots,\varpi^{a_n})v
\]
where $a_r-a_{r+1}>N$ and $u,v\in K_0$. We hope to show $g\notin S$. For any $M>0$, we claim
\[
(K_0\cap gK_Ng^{-1})\cdot K_m=(K_0\cap hK_Nh^{-1})\cdot K_m
\]
so in particular
\[
\lambda^{K\cap gK_Ng^{-1}}=\lambda^{K\cap hK_Nh^{-1}}
\]
if
\[
h=u(\varpi^{a_1+M},\cdots,\varpi^{a_r+M},\varpi^{a_{r+1}},\cdots,\varpi^{a_n})v.
\] Indeed,
\[
(K_0\cap hK_Nh^{-1})\cdot K_m=(1_{n\times n}+(\p^{a_{ij}}))\cdot K_m,
\]
where
\[
a_{ij}\colonequals \begin{cases}\max\{N+a_i-a_j,0\}&\text{if }i,j\le r\text{ or }i,j>r,\\
\max\{N+a_i-a_j+M,0\}&\text{if }i\le r<j,\\
\max\{N+a_i-a_j-M,0\}&\text{if }i> r\ge j.
\end{cases}
\]
We see that if $i>r\ge j$ then
\[
N+a_i-a_j-M\le N+a_r-a_{r+1}\le 0
\]
so $a_{ij}=0$, and otherwise $a_{ij}\ge N$, so
\[
(K_0\cap hK_Nh^{-1})K_m=\begin{pmatrix}
1_{r\times r}&0_{r\times (n-r)}\\
*&1_{(n-r)\times(n-r)}
\end{pmatrix}\cdot K_m,
\]
which is independent of the choice of $M$.

In particular, $g\in S$ if and only if $h\in S$. Thus, if $g\in S$ then
\[
ZK_0\backslash S/K_N\supset ZK_0\backslash \bigcup_{M\ge0}ZK_0(\varpi^{a_1+M},\cdots,\varpi^{a_r+M},\varpi^{a_{r+1}},\cdots,\varpi^{a_n})K_0/K_N,
\]
which is clearly infinite, contradicting $ZK_0\backslash S/K_N$ being finite.

Finally, putting together \eqref{lower-bound}, \eqref{upper-bound}, and Lemma~\ref{coset-size}, we obtain:
\begin{align*}
    \dim(\rho^{K_N})&\sim|\bK\backslash S/K_N|\\
    &\sim\sum_{\substack{a_1\ge\cdots\ge a_n\\a_i-a_{i+1}<N}}q^{\sum_{i<j}\min\{a_i-a_j,N\}}\\
    &\sim q^{\frac{n(n-1)}2N},
\end{align*}
concluding the proof of Proposition~\ref{dim:cuspidal} that for $\rho$ a supercuspidal representation of $\GL_n(F)$,
\[
\dim_{GK}(\rho)=\frac{n(n-1)}2.\qedhere
\]
\end{proof}
\begin{rmk} The technique in the proof of Proposition~\ref{dim:cuspidal} is quite general, once one knows a representation is induced from a compact open modulo center subgroup. For instance, \cite[Thm~2.1]{Bush-Kutz} shows any supercuspidal representation $\rho$ of $\SL_n(F)$ is induced from a compact open subgroup, so again,
\[
\dim(\rho^{K_N\cap\SL_n(F)})\sim q^{\frac{n(n-1)}2N}.
\]
A more precise bound is given in Appendix~\ref{section-sl-n}.
\end{rmk}

\section{Level zero representations of $\GL_2$ and $\GL_3$}\label{appendix-level-0}
An irreducible representation $\pi$ of $G$ is \emph{of level zero} if $\pi^{K_1}\ne0$. Such representations which are cuspidal are classified by the following:

\begin{prop}[{\cite[Section~1.2]{Bushnell}}]
Let $\rho$ be an irreducible cuspidal representation of level zero. Then, there exists a representation $\Lambda$ of $K_0Z$ such that $\Lambda|_{K_0}$ is inflated from a cuspidal representation of $\GL_n(k)$, such that
\[
\rho\cong\cInd_{K_0Z}^G\Lambda.
\]
\end{prop}

The growth polynomial of level zero representations may be calculated using the techniques in the proof of Proposition~\ref{dim:cuspidal}. A general formula for $\dim(\rho^{K_N})$, for $\rho$ supercuspidal of level zero is an immediate consequence of Proposition~\ref{murn}:

\begin{cor} If $\rho$ is a supercuspidal level zero representation, then
\[
G_\rho(X)=\sum_{\cO}(-1)^{n+r}n(r-1)!w_\cO^{-1}\begin{bmatrix}n\\n_1,\dots,n_r\end{bmatrix}_qX^{\frac12\dim\cO},
\]
where the nilpotent orbit $\cO$ correspond to the partition $n=n_1+\cdots+n_r$, and $w_\cO$ is the number of permutation of $r$ letters which fix the $r$-tuple $(n_1,\dots,n_r)$.
\end{cor}

We offer an alternate, \emph{more direct} calculation for level zero representations for small $n$.

Let $\rho=\cInd_{K_0Z}^G\Lambda$ be an irreducible supercuspidal representation of level zero.
By Lemma~\ref{fixed},
\begin{align*}
    \rho^{K_N}&=\bigoplus_{g\in K_0Z\backslash G/K_N}\Lambda^{K_0Z\cap gK_Ng^{-1}}\\
    &=\bigoplus_{g\in K_0Z\backslash G/K_N}\sigma^{(K_0\cap gK_Ng^{-1})/K_1},
\end{align*}
where $\Lambda|_K$ is inflated from the representation $\sigma$ of $\GL_n(k)$. Since $K_0Z$ normalizes $K_N$, only the images of $g$ in $K_0Z\backslash G/K_0$ matter. That is, we may suppose $g$ is of the form $(\varpi^{a_1},\dots,\varpi^{a_n})$ for integers $a_1\ge\cdots\ge a_n=0$. Thus,
\begin{equation}\label{level0-formula}
\dim(\rho^{K_N})=\sum_{a_1\ge\cdots\ge a_n=0}|K_0\backslash K_0(\varpi^{a_1},\dots,\varpi^{a_n})K_0/K_N|\cdot\dim(\sigma^{(K_0\cap gK_Ng^{-1})/K_1}).
\end{equation}

This gives a polynomial expression for the dimension $\dim(\rho^{K_N})$. To illustrate the techniques, we will give an explicit formula when $n=2,3$ (in the case $n=2$, this is a direct confirmation of Example~\ref{gl2-dim}):

\begin{prop}
When $\rho$ is a supercuspidal level zero representation of $\GL_2$,
\[
G_\rho(X)=[2!]_qX-2
\]
and when $\rho$ is a supercuspidal level zero representation of $\GL_3$,
\[
G_\rho(X)=[3!]_qX^3-3[3]_qX^2+3,
\]
where $X=q^{N-1}$, and $[n]_q\colonequals \frac{q^n-1}{q-1}$, $[n!]_q\colonequals [1]_q\cdots[n]_q$.
\end{prop}
\begin{proof}

When $n=2$ and $g=(\varpi^a,1)\in K_0Z\backslash G/K_0$ with $a\ge0$, then
\[
(K_0\cap gK_Ng^{-1})/K_1=\begin{pmatrix}1&0\\\p^{\max\{0,N-a\}}&1\end{pmatrix}/K_1=\begin{cases}1&\text{if }N>a\\
\begin{pmatrix}1&0\\{*}&1\end{pmatrix}&\text{if }N\le a,
\end{cases}
\]
so since $\sigma$ is cuspidal,
\[
\sigma^{(K_0\cap gK_Ng^{-1})/K_1}=\begin{cases}
\sigma &\text{if }N>a\\
0& \text{if }N\le a.
\end{cases}
\]
Thus, since $\dim(\sigma)=q-1$ (see \cite[Thm~6.4]{Bush-Henn}), by \eqref{level0-formula} and Lemma~\ref{coset-size} we have:
\begin{align*}
    \dim(\rho^{K_N})&=\sum_{0\le a<N}(q-1)|K_0\backslash K_0(\varpi^a,1)K_0/K_N|\\
    &=(q-1)+\sum_{0<a<N}(q-1)\cdot q^{a-1}(q+1)\\
    &=(q+1)q^{N-1}-2.
\end{align*}
When $n=3$ and $g=(\varpi^a,\varpi^b,1)\in K_0Z\backslash G/K_0$ with $a\ge b\ge0$ then
\[
(K_0\cap gK_Ng^{-1})/K_1=\begin{cases}1&\text{if }N>a\\
\begin{pmatrix}1\\&1\\{*}&&1\end{pmatrix}&\text{if }a\ge N,a-b,b<N\\
\text{Levi subgroup}&\text{otherwise}.
\end{cases}
\]
In the last case, $\sigma^{(K_0\cap gK_Ng^{-1})/K_1}=0$ since $\sigma$ is cuspidal. In the first case, $\sigma^{(K_0\cap gK_Ng^{-1})/K_1}=\sigma$. For the second case, decompose $\sigma$ as follows:
\[
\sigma\cong\bigoplus_{\chi: N\to\C^\times}\chi^{m_\chi},
\]
where $N\colonequals \begin{pmatrix}1&\\&1\\{*}&{*}&1\end{pmatrix}$ is a nilpotent subgroup. Since $\sigma$ is cuspidal, $m_1=0$, and the conjugation action by $\GL_2(k)\times k^\times$ on the subgroup $N$ permutes the other characters transitively so $m_\chi$ is constant for nontrivial $\chi$. Thus, in fact $\sigma\cong\bigoplus_{\chi\ne1}\chi^{\oplus d/(q^2-1)}$, where $d=(q-1)(q^2-1)$ is the dimension of $\sigma$.

Thus, $\sigma^{\begin{pmatrix}1\\&1\\{*}&&1\end{pmatrix}}$ is of dimension $d/(q+1)$. Combining,
\begin{align*}
    \dim(\rho^{K_N})=d\sum_{N>a\ge b\ge0}&|K_0\backslash K_0(\varpi^a,\varpi^b,1)K_0/K_N|\\&+\frac{d}{q+1}\sum_{\substack{a\ge b\ge 0\\a-b,b<N,a\ge N}}|K_0\backslash K_0(\varpi^a,\varpi^b,1)K_0/K_N|.
\end{align*}
For each of these sums, we will run casework on what partition $\{a,b,0\}$ corresponds to and apply Lemma~\ref{coset-size}:
\begin{itemize}
    \item $3=3$:
    \[
    d\sum_{a=b=0}|K_0\backslash K_0(\varpi^a,\varpi^b,1)K_0/K_N|=d.
    \]
    \item $3=2+1$:
    \begin{align*}
    d\sum_{N>a=b>0}|K_0\backslash K_0(\varpi^a,\varpi^b,1)K_0/K_N|&=d\sum_{N>a>0}q^{-2+2a}[3]_q\\
    &=\frac{d[3]_q}{q^2-1}(q^{2N-2}-1).
    \end{align*}
    \item $3=1+2$: the formula is the same as in $3=1+2$.
    \item $3=1+1+1$:
    \begin{align*}
        d\sum_{N>a>b>0}|K_0\backslash K_0(\varpi^a,\varpi^b,1)K_0/K_N|&=d\sum_{N>a>0}(a-1)q^{-3+2a}[3!]_q\\
        &=\frac{d[3!]_q}{(q^2-1)^2}((N-2)q^{2N-1}-(N-1)q^{2N-3}+q).
    \end{align*}
\end{itemize} 
For the second term, only the partition $3=1+1+1$ appears, and evaluates as (this is the main term):
\begin{align*}
    \frac{d}{q+1}\sum_{\substack{a>b>0\\a-b,b<N,a\ge N}}q^{-3+a+N}[3!]_q&=\frac{d}{q+1}\sum_{a=N}^{2N-2}(2N-a-1)q^{-3+a+N}[3!]_q\\
    &=\frac{d[3!]_q}{(q+1)(q-1)^2}(q^{3N-3}-Nq^{2N-2}+(N-1)q^{2N-3}).
\end{align*}
Combining everything, and using $d=(q-1)(q^2-1)$, we obtain:
\[
\dim(\rho^{K_N})=[3!]_q\cdot q^{3N-3}-3[3]_qq^{2N-2}+3.\qedhere
\]
\end{proof}

\begin{rmk}
Combining the sums corresponding to the partition $3=1+1+1$ from both terms gives a polynomial in $q^{N-1}$, suggesting a clever reorganization of sums may help simplify and generalize the calculation.
\end{rmk}

\end{document}